\documentclass[a4paper, oneside]{article}

\usepackage{amsmath}					%
\usepackage{amssymb}					%
\usepackage{amsthm}					%
\usepackage{appendix}					%
\usepackage{color}						%
\usepackage{enumerate}				%
\usepackage{hyperref}					%
\usepackage{mathrsfs}					%
\usepackage{pdfsync}					%
\usepackage{titlesec}					%

\titleformat{\section}[hang]							%
{\bfseries\large}{\thesection.}{0.5em}{}[]				%
\titlespacing*{\section}{0em}{2em}{1.5em}				%

\titleformat{\subsection}[runin]						%
{\bfseries\normalsize}{\thesubsection.}{0em}{\ }[.]		%

\theoremstyle{plain}								%
\newtheorem{thm}{Theorem}[section]					%
\newtheorem{prop}[thm]{Proposition}					%
\newtheorem{lem}[thm]{Lemma}						%

\theoremstyle{definition}							%
\newtheorem{defn}[thm]{Definition}					%

\theoremstyle{remark}								%
\newtheorem{rem}[thm]{Remark}						%

\numberwithin{equation}{section}						%

\DeclareMathOperator*{\essinf}{essinf}				
\DeclareMathOperator*{\esslim}{esslim}				
\DeclareMathOperator{\sgn}{sgn}						

\newcommand{\Romannum}[1]{\uppercase\expandafter{\romannumeral#1\relax}}
\newcommand{\romannum}[1]{\romannumeral#1\relax}

\def \ep {\epsilon}			
\def \ve {\varepsilon}			
\def \vf {\varphi}				

\def \bE {\mathbb E}			
\def \bP {\mathbb P}			
\def \bR {\mathbb R}			
\def \bZ {\mathbb Z}			

\def \cA {\mathcal A}			%
\def \cC {\mathcal C}			%
\def \cS {\mathcal S}			%

\def \fD {\mathfrak D}			
\def \fQ {\mathfrak Q}			

\def \CS {Cauchy--Schwarz }

\def \exc {\mathrm{exc}}
\def \LS {\mathrm{LS}}
\def \Ls {L_\mathrm{SS}}
\def \Lta {L_\mathrm{TAS}}
\def \wts {\widetilde\sigma}

\begin{document}

\pagestyle{plain}
\pagenumbering{arabic}
\bibliographystyle{plain}

\title{Hydrodynamic limit for asymmetric simple exclusion \\with accelerated boundaries}
\author{\textsc{Lu XU}}
\date{}
\maketitle



\begin{abstract}
We consider the asymmetric simple exclusion process (ASEP) on the one-dimensional finite lattice $\{1,2,\ldots,N\}$.
The particles can be created/annihilated at the boundaries with given rates.
These rates are $L^\infty$ functions of time and are independent of the jump rates in the bulk (cf. \cite{Baha12}).
The boundary dynamics is modified by a factor $N^\theta$ with $\theta>0$.
We study the hydrodynamic limit for the particle density profile under the hyperbolic space-time scale.
The macroscopic equation is given by (inviscid) Burgers equation with Dirichlet type boundaries that is characterized by the boundary entropy \cite{Otto96}.
A grading scheme is developed to control the formulation of boundary layers on the microscopic level.

\bigskip

Nous consid\'{e}rons le processus d'exclusion simple asym\'{e}trique (ASEP) sur le r\'{e}seau fini unidimensionnel $\{1,2,\ldots ,N\}$.
Les particules peuvent \^{e}tre cr\'{e}\'{e}es/annihil\'{e}es sur les points de fronti\`{e}re avec des taux qui sont des fonctions $L^{\infty}$ du temps et sont ind\'{e}pendants des taux de saut \`{a} l'int\'{e}rieur du syst\`{e}me (cf. \cite{Baha12}).
La dynamique des bords est modifi\'{e}e d'un facteur $N^{\theta}$ avec $\theta >0$.
Nous \'{e}tudions la limite hydrodynamique du profil de densit\'{e} des particules sous l'\'{e}chelle espace-temps hyperbolique.
L'\'{e}quation macroscopique est donn\'{e}e par l'\'{e}quation (non visqueuse) de Burgers avec des conditions aux limites caract\'{e}ris\'{e}es par l'entropie du processus sur les points de fronti\`{e}re \cite{Otto96}.
Un sch\'{e}ma adapt\'{e} est d\'{e}velopp\'{e} pour contr\^{o}ler la formulation des couches limites au niveau microscopique.
\end{abstract}

\section{Introduction}

Exclusion process in contact with reservoirs is a widely studied model in non-equilibrium statistical physics.
In its dynamics, each particle performs a random walk on the finite lattice $\{1,2,\ldots,N\}^d$ in accordance with the exclusion rule: if the target site is occupied, the corresponding jump will be suppressed.
At boundaries, the particles can be removed from or added into the system with given rates.
For symmetric dynamics, the hydrodynamic limit for the particle density is given by diffusion equation with Dirichlet boundary conditions imposed by the reservoirs, see \cite{ELS90,ELS91,KipnisL99,LMO08} and the reference therein.
If the bulk dynamics performs weak asymmetry, viscous Burgers equation is obtained in the limit \cite{BEMM04,BLM09}.
More recent works succeed in getting varies types of boundary conditions by modifying the strength of reservoirs, for instance \cite{BMNS17,CapiG21,FGN19}.

The situation differs drastically when a strong drift exists. 
Liggett \cite{Liggett75} first introduced the asymmetric simple exclusion process (ASEP) with open boundaries as a tool for studying the dynamics of the infinite system. 
For one-dimensional case, the stationary states are studied in \cite{BCEPR06,DEHP93,USW04}.
They are given by uniform bulk density which is determined by the reservoir rates through three phases: the high-density phase, the low-density phase, and the max-current phase.
The high-density phase and the low-density phase intersect at a critical line, where the bulk density performs a randomly located shock.
These stationary states satisfy the variational property \cite{PopkovS99}: the current is minimized if $\rho_-<\rho_+$ (drift up-hill) and is maximized if $\rho_->\rho_+$ (drift down-hill), where $\rho_\pm$ are the left and right boundary rates, respectively.
The hydrodynamic limit for ASEP on $\bZ^d$ is obtained in \cite{Reza91} under the hyperbolic space-time scale.
The celebrated result of \cite{Baha12} proves that the hydrodynamic limit for a class of boundary driven asymmetric particle systems is given by the unique entropy solution to the initial--boundary problem of hyperbolic conservation law with boundary conditions introduced in \cite{BLN79}.
Different from the Dirichlet conditions in diffusion equation, these boundary conditions do not fix the boundary value of the solution.
Instead, they impose a set of possible boundary values depending on the boundary data.
The proof in \cite{Baha12} relies on a special choice of reservoir densities that approximates optimally the infinite dynamics (see Remark \ref{rem:liggett} for details), thus allows the application of an effective coupling method.
Also note that all these previous results are restricted to reservoir densities that do not depend on time.

Recently, the ASEP connected to general reservoirs whose densities vary with time and are regulated by a factor $N^\theta$ is studied \cite{DMOX22,Xu22}.
When $\theta<0$, the reservoirs are slowed down and the hydrodynamic limit is given by the entropy solution to Burgers equation subjected to degenerated boundary traces \cite{Xu22}.
The proof in \cite{Xu22} relies on the vanishing boundary drift created by the slow reservoirs.
When $\theta>0$ and the reservoir densities are smooth functions of time, \cite{DMOX22} obtained the \emph{quasi-static limit}: under the time scale longer than the hyperbolic one, the particle density is determined by the quasi-stationary solution \cite{MOX21} to Burgers equation, that is, the stationary solution associated to dynamical boundary data.
The phase diagram \cite{DEHP93} and the variational property \cite{PopkovS99} are also reproduced, see Section \ref{subsec:qs}.
The proof in \cite{DMOX22} is established upon a fine estimate on the boundary entropy production for the balanced dynamics where the reservoirs impose the same density at left and right boundaries.
Then, the results are extended to unbalanced case by using a coupling method, which relies on the fact that the quasi-static dynamics always stays close to stationary states.

In the present article, we consider the ASEP in contact to reversible reservoirs whose densities are general $L^\infty$ functions of time and are accelerated by $N^\theta$ with $\theta>0$.
We prove in Theorem \ref{thm:hyd} that, under the hyperbolic time scale, the density profile evolves with the entropy solution to Burgers equation with boundary conditions characterized by boundary entropy flux pairs \cite{Otto96}.
Since the strengthened reservoirs dominate the drift, the \emph{formal} boundary values in the macroscopic equation coincide with the reversible densities of the reservoirs, see Remark \ref{rem:bd}.

The main difference from the previous works is that, as stated above, the coupling method used in \cite{Baha12} is not available in this general setting of reservoir dynamics.
Furthermore, since the dynamics does not have enough time to relax to the stationary state, the strategy in \cite{DMOX22} is not applicable.
Our new approach contains two main steps.
First, we apply the stochastic compensated compactness developed in \cite{Fritz04} to show the existence of a random $L^\infty$ function as the limit density profile.
Then, we introduce a grading scheme, which helps to extend the estimate on the boundary entropy production proved in \cite{DMOX22} to general unbalanced dynamics.
The key point in this step is the auxiliary functions defined in \eqref{eq:auxiliary}, which slices the divergent boundary integral into pieces with an exact cancellation, see Lemma \ref{lem:esti-s-bd} and \ref{lem:esti-b}.
The proof is finally closed by the uniqueness of the entropy solution established for scalar conservation law on a bounded domain \cite{MNRR96,Otto96}.

The technique developed in this article does not rely on the attractiveness of the dynamics, thus it is expected to be applicable to non-attractive systems, such as non-attractive asymmetric zero-range process.
Meanwhile, the hydrodynamic limit remains an open problem for $\theta =0$ except the regime treated in \cite{Baha12}.
We expect the macroscopic boundary conditions to be given by the vanishing viscosity limit obtained in \cite[Theorem 2.12]{Xu22}.

\bigskip
\noindent\textbf{Keywords.}
Asymmetric simple exclusion process, Open boundary, Hydrodynamic limit, Entropy solution. 

\bigskip
\noindent\textbf{Acknowledgments.}
The author greatly thanks Stefano Olla for suggesting him to study this problem as well as his inspiring advices. 

\section{One-dimensional ASEP with open boundaries}
\label{sec:asep}

The one-dimensional asymmetric simple exclusion process with open boundaries is a Markov process defined on the configuration space
\begin{align}
  \Omega_N := \big\{\eta=(\eta_1,\eta_2,\ldots,\eta_N),\ \eta_i\in\{0,1\}\big\}, \quad N\ge2.
\end{align}
Its infinitesimal generator is given by
\begin{align}
  L_N = p\Lta + \sigma\Ls + L_- + L_+,
\end{align}
where $p$ and $\sigma$ are strictly positive constants.
Here $\Lta$ and $\Ls$ generate respectively the totally asymmetric exclusion and symmetric exclusion in bulk:
\begin{equation}
\label{eq:gen-exc}
  \begin{split}
    \Lta f &:= \sum_{i=1}^{N-1} \eta_i(1-\eta_{i+1})\big[f(\eta^{i,i+1})-f(\eta)\big],\\
    \Ls f &:= \sum_{i=1}^{N-1}\big[f(\eta^{i,i+1})-f(\eta)\big],
  \end{split}
\end{equation}
for any function $f$ on $\Omega_N$, where $\eta^{i,i+1}$ is the configuration obtained from $\eta$ upon exchanging $\eta_i$ and $\eta_{i+1}$.
The boundary generators $L_\pm$ generate the reservoirs that creates/annihilates particles at $i=1$ and $i=N$ with given rates $\alpha$, $\beta$, $\gamma$, $\delta>0$:
\begin{equation}
\label{eq:gen-bd}
  \begin{split}
    L_-f &:= [\alpha({1-\eta_1})+\gamma\eta_1]\big[f(\eta^1)-f(\eta)\big],\\
    L_+f &:= [\delta({1-\eta_N})+\beta\eta_N]\big[f(\eta^N)-f(\eta)\big],
  \end{split}
\end{equation}
where $\eta^i$ is the configuration obtained from $\eta$ by shifting the status at site $i$ from $\eta_i$ to $1-\eta_i$. 
Note that $L_\pm$ are respectively reversible for the Bernoulli measure with densities
\begin{align}
  \rho_-:=\frac\alpha{\alpha+\gamma}, \quad \rho_+:=\frac\delta{\beta+\delta}.
\end{align}

\begin{rem}[Liggett's boundaries]
\label{rem:liggett}
In \cite{Baha12,DEHP93,Liggett75} a specific choice of boundary parameters $\alpha$, $\beta$, $\gamma$, $\delta$ are adopted.
Given some $\bar\rho_\pm\in[0,1]$, define
\begin{equation}
\label{eq:liggett}
  \begin{aligned}
    &\alpha:=(p+\sigma)\bar\rho_-, \quad \gamma:=\sigma(1-\bar\rho_-),\\
    &\beta:=(p+\sigma)(1-\bar\rho_+), \quad \delta:=\sigma\bar\rho_+.
  \end{aligned}
\end{equation}
This choice corresponds to the projection on the finite interval $\{1,\ldots,N\}$ of the infinite ASEP dynamics with jump rate $p+\sigma$ to the right and $\sigma$ to the left, where particles are distributed according to Bernoulli measure with density $\bar\rho_-$ on $\{0, -1, \ldots\}$ and density $\bar\rho_+$ on $\{N+1,N+2,\ldots\}$.
If furthermore $\bar\rho_-=\bar\rho_+=\bar\rho$, the homogeneous Bernoulli measure with density $\bar\rho$ is stationary with respect to the dynamics. 
Observe that due to the non-vanishing current, the stationary densities differ from the densities of the reservoirs. 
\end{rem}

\subsection{Parameters}
We want to treat boundary rates that vary non-smoothly with time.
Consider positive functions $\alpha$, $\beta$, $\gamma$, $\delta \in L^\infty(\bR_+)$ and define operators $L_{\pm,t}$ by \eqref{eq:gen-bd} with $(\alpha,\beta,\gamma,\delta)(t)$.
Also pick two sequences $\sigma_N$ and $\wts_N$ satisfying that
\begin{align}
\label{eq:assp}
  \lim_{N\to\infty} N^{\frac57}\sigma_N^{-1} = \lim_{N\to\infty} N^{-1}\sigma_N = 0,
  \quad \lim_{N\to\infty} \wts_N = \infty.
\end{align}
Define the infinitesimal generator $L_{N,t}$ as
\begin{align}
\label{eq:gen}
  L_{N,t} = p\Lta + \sigma_N\Ls + \wts_N\big(L_{-,t} + L_{+,t}\big).
\end{align}
In \eqref{eq:gen}, the symmetric exclusion is accelerated by $\sigma_N$ in order to enhance convergence to local equilibrium at a mesoscopic scale, while the reservoirs are strengthened by $\wts_N$ to fix the densities at $1$ and $N$, see Proposition \ref{prop:bd-one-block}.
The technical assumption $\sigma_N \gg N^\frac57$ allows us to choose the specific mesoscopic scale in \eqref{eq:assp-k}.
Moreover, this dynamic is NOT the so called \emph{weakly asymmetric exclusion} since the asymmetry is strong enough to survive in the hyperbolic space-time scaling limit. 
From a macroscopic point of view, $\sigma_N\Ls$ plays the role of vanishing viscosity in the hyperbolic conservation law. 

\subsection{Hydrodynamic limit}
Let $\eta(t)=(\eta_1(t),\ldots,\eta_N(t)) \in \Omega_N$ be the Markov process generated by $NL_{N,t}$ and some initial distribution $\mu_{N,0}$.
Assume that for all $\phi \in \cC([0,1])$ and $\ve>0$,
\begin{align}
\label{eq:assp-initial}
  \lim_{N\to\infty} \mu_{N,0} \left\{ \left| \frac1N\sum_{i=1}^N \phi \left( \frac iN \right) \eta_i - \int_0^1 \phi(x)u_0(x)dx \right| > \ve \right\} = 0
\end{align}
with some $u_0 \in L^\infty([0,1])$.
Let $\chi_{i,N}$ be the indicator function 
\begin{align}
\label{eq:indicator}
  \chi_{i,N}(x) :=
  \mathbf1_{\left\{\left[\frac iN-\frac1{2N}, \frac iN+\frac1{2N}\right) \cap [0,1]\right\}}(x),
  \quad \forall\,x\in[0,1].
\end{align}
For each $N$, define the empirical density $\zeta_N = \zeta_N(t,x)$ as
\begin{align}
\label{eq:empirical}
  \zeta_N(t,x) := \sum_{i=1}^N \eta_i(t)\chi_{i,N}(x), \quad (t,x) \in [0,+\infty)\times[0,1].
\end{align}
Our aim is to show that, as $N \to \infty$, $\zeta_N$ converges to the $L^\infty$ weak entropy solution to the one-dimensional inviscid Burgers equation
\begin{align}
\label{eq:cl1}
  \partial_tu(t,x) + p\partial_xJ(u(t,x)) = 0, \quad x\in (0,1), \quad J(u) = u(1-u),
\end{align}
with $L^\infty$ boundary and initial data formally given by
\begin{align}
\label{eq:cl2}
  u|_{x=0}=\rho_-(t):=\frac{\alpha(t)}{\alpha(t)+\gamma(t)}, \quad u|_{x=1}=\rho_+(t):=\frac{\delta(t)}{\beta(t)+\delta(t)}, \quad u|_{t=0}=u_0.
\end{align}
Notice that \eqref{eq:cl2} is not satisfied in the classical sense.
Instead, it allows discontinuity at the boundaries, usually called the boundary layers, by imposing a set of possible boundary values relevant to $\rho_\pm(t)$.
Rigorous definition of the entropy solution to \eqref{eq:cl1}--\eqref{eq:cl2} is stated in Section \ref{sec:burgers}. 
Our main result is stated below. 

\begin{thm}[Hyperbolic hydrodynamic limit]
\label{thm:hyd}
Suppose \eqref{eq:assp}, \eqref{eq:assp-initial} and that 
\begin{align}
\label{eq:assp-bd}
  \essinf_{0 \le t \le T} \big\{\alpha(t), \beta(t), \gamma(t), \delta(t)\big\} > 0
\end{align}
for some $T>0$. 
For any $\psi\in\cC(\bR^2)$, the following limit holds in probability: 
\begin{align}
\label{eq:convergence}
  \lim_{N\to\infty} \int_0^T \frac1N\sum_{i=1}^N \psi \left( t,\frac iN \right)\eta_i(t)dt = \int_0^T \!\!\! \int_0^1 \psi(t,x)\rho(t,x)dx\,dt, 
\end{align}
where $\rho=\rho(t,x)$ is the entropy solution to \eqref{eq:cl1}--\eqref{eq:cl2}. 
\end{thm}

\begin{rem}[Boundary values]
\label{rem:bd}
Without the factor $\wts_N$ speeding up the reservoirs, we conjecture the boundary values are given by some $\tilde\rho_-(t)=\tilde\rho_-(p,\alpha,\gamma)$ and $\tilde\rho_+(t)=\tilde\rho_+(p,\beta,\delta)$.
Due to the non-zero drift from $0$ to $1$, it is expected that $\tilde\rho_-<\rho_-$, $\tilde\rho_+>\rho_+$.
For the particular case where $(\alpha,\beta,\gamma,\delta)$ are determined through \eqref{eq:liggett} with parameters $\bar\rho_\pm$ that is constant in time, the boundary conditions are given by \cite{Baha12}
\begin{align}
  u(t,0)=\bar\rho_-, \quad u(t,1)=\bar\rho_+. 
\end{align}
The result for general case remains an open problem.
\end{rem}

\subsection{A few words on the quasi-stationary profile}
\label{subsec:qs}

The \emph{quasi-stationary} solution to \eqref{eq:cl1}--\eqref{eq:cl2} is defined as the limit, when $\ep\to0^+$, of the entropy solution to
\begin{align}
  \ep\partial_tu^\ep+p\partial_xJ(u^\ep)=0, \quad u^\ep|_{x=0}=\rho_-, \quad u^\ep|_{x=1}=\rho_+, \quad u^\ep|_{t=0}=u_0.
\end{align}
Excluding the critical phase given by
\begin{align}
  (\rho_-,\rho_+) \in \big\{(a,b)\in(0,1)^2,\ a<1/2,\ a+b=1\big\},
\end{align}
this limit is constant in $x$ and is explicitly given in \cite{MOX21}:
\begin{align}
  \bar u(t) =
  \begin{cases}
    \rho_-(t), &\text{if } \rho_-(t)<1/2,\ \rho_-(t)<1-\rho_+(t),\\
    \rho_+(t), &\text{if } \rho_+(t)>1/2,\ \rho_+(t)>1-\rho_-(t),\\
    1/2, &\text{if } \rho_-(t)\ge1/2,\ \rho_+(t)\le1/2.
  \end{cases}
\end{align}
This extends the variational property in \cite{PopkovS99} to a quasi-static sense:
\begin{align}
  J(\bar u(t)) =
  \begin{cases}
    \inf \{J(\rho); \rho\in[\rho_-(t),\rho_+(t)]\}, &\text{if}\ \rho_-(t)\le\rho_+(t),\\
    \sup \{J(\rho); \rho\in[\rho_+(t),\rho_-(t)]\}, &\text{if}\ \rho_-(t)>\rho_+(t).
  \end{cases}
\end{align}

The quasi-stationary profile can be directly derived from the microscopic dynamics under the quasi-static time scale.
Precisely speaking, for $a>0$ let $N^{1+a}L_{N,t}$ generate the process $\eta'(t)\in\Omega_N$.
Suppose $\zeta'_N$ to be the empirical density of $\eta'$ defined in \eqref{eq:empirical}.
Then $\zeta'_N$ converges to $\bar u$, as $N\to\infty$ in the sense of \eqref{eq:convergence}.
This quasi-static limit has been proved in \cite[Theorem 3.3]{DMOX22} when the boundary data are smooth functions of $t$ and $\rho_--\rho_+$ keeps its sign within $[0,T]$.
We expect the same result to hold for general non-smooth boundary data.

\section{Entropy solution on bounded domain}
\label{sec:burgers}

Fix two measurable functions $\rho_\pm: [0,\infty)\to[0,1]$ and recall the initial--boundary problem in \eqref{eq:cl1}--\eqref{eq:cl2}.
We say that $(f,q) \in \cC^2([0,1];\bR^2)$ is a Lax entropy flux pair associated to $J$ if
\begin{align}
  q'(u)=J'(u)f'(u)=(1-2u)f'(u), \quad \forall\,u\in[0,1].
\end{align}
The pair $(f,q)$ is called convex if $f''(u)\ge0$ for $u\in[0,1]$.
To characterise the boundary conditions, Otto \cite{Otto96} introduced the boundary entropy.

\begin{defn}
\label{def:bd-ent}
A \emph{boundary entropy flux pair} associated to \eqref{eq:cl1} is a couple of $\cC^2$ functions $(F,Q): [0,1]\times\bR \to \bR^2$ such that
\begin{align}
  F(w,w)=Q(w,w)=\partial_uF(w,w)=0, \quad \forall\,(u,w) \in [0,1]\times\bR,
\end{align}
and $(F,Q)(\cdot,w)$ is a Lax entropy flux pair for each $w\in\bR$.
Moreover, we say that the pair $(F,Q)$ is convex if $F(u,w)$ is convex in $u$ for all $w\in\bR$.
\end{defn}

The entropy solution to \eqref{eq:cl1}--\eqref{eq:cl2} is a function $\rho \in L^\infty(\bR_+\times[0,1])$ satisfying
\begin{itemize}
\item[\romannum1.] the entropy inequality in the weak sense
\begin{align}
\label{eq:ent-sol1}
  \int_0^\infty \int_0^1 f(\rho)\partial_t\vf\,dx\,dt + p\int_0^\infty \int_0^1 q(\rho)\partial_x\vf\,dx\,dt \ge 0,
\end{align}
for all convex entropy flux pair $(f,q)$ and positive function $\vf\in\cC^\infty(\bR_+\times(0,1))$ with compact support;
\item[\romannum2.] the initial condition in the sense
\begin{align}
\label{eq:ent-sol2}
  \esslim_{t\to0^+} \int_0^1 |\rho(t,x)-u_0(x)|dx = 0;
\end{align}
\item[\romannum3.] the boundary conditions introduced by Otto \cite{Otto96}
\begin{equation}
\label{eq:ent-sol3}
  \begin{split}
   & \esslim_{r\to 0^+} \int_0^\infty Q\big(\rho(t,r),\rho_-(t)\big) \phi(t)dt \le 0,\\
   & \esslim_{r\to 0^+} \int_0^\infty Q\big(\rho(t,1-r),\rho_+(t)\big) \phi(t)dt \ge 0,
  \end{split}
\end{equation}
for all convex boundary entropy flux $Q$ and positive function $\phi\in\cC(\bR_+)$ with compact support.
\end{itemize}
Observe that the equality in \eqref{eq:ent-sol1} is reached for $(f,q)=(u,J(u))$, hence entropy solution is in particular a weak solution.
It is not hard to see that $\rho(t,x) \in [0,1]$ almost surely.

\begin{rem}[Boundary traces]
In the case $\rho$ is of bounded variation, \eqref{eq:ent-sol3} coincides with the
Bardos--Leroux--N\'ed\'elec condition \cite{BLN79}.
Later, Vasseur \cite{Vass01} proved that $\rho$ has strong boundary traces even in $L^\infty$ regime, so Bardos--Leroux--N\'ed\'elec's condition is generally valid.
Nevertheless, we use an alternative expression that does not contain the explicit information of the boundary traces, see \eqref{eq:ent-sol} for detail.
\end{rem}

\section{Proof of Theorem \ref{thm:hyd}}
\label{sec:proof}

Fix some $T>0$ and denote $\Sigma_T=[0,T]\times[0,1]$.
By a Young measure on $\Sigma_T\times[0,1]$ we mean a positive measure $\nu$ on $\Sigma_T\times[0,1]$ such that
\begin{align}
\label{eq:young-condition}
  \nu(A\times[0,1])=\mu_\mathrm{Leb}(A), \quad \forall\,\text{Borel measurable set}\ A\subseteq\Sigma_T,
\end{align}
where $\mu_\mathrm{Leb}$ is the Lebesgue measure.
Denote by $\mathcal Y$ the set of all Young measures, endowed with the weak topology, i.e., a sequence $\nu^n \to \nu$ if and only if
\begin{align}
  \lim_{n\to\infty} \int f\,d\nu^n = \int f\,d\nu, \quad \forall\,f\in\cC(\Sigma_T\times[0,1]).
\end{align}
The condition \eqref{eq:young-condition} being stable under weak convergence, so $\mathcal Y$ is closed.
Since $\Sigma_T\times[0,1]$ is a Polish space, also is $\mathcal Y$.
Moreover, Prokhorov's theorem yields that $\mathcal Y$ is compact.

By the slicing measure theorem (cf. \cite[pp.14, Theorem 1.10]{Evans90}), $\nu\in\mathcal Y$ can be interpreted as a family $\{\nu_{t,x}; (t,x)\in\Sigma_T\}$ of probability measures on $[0,1]$, such that
\begin{align}
  \int f\,d\nu = \iint_{\Sigma_T} dx\,dt\int_0^1 f(t,x,u) \nu_{t,x}(du), \quad \forall\,f\in\cC(\Sigma_T\times[0,1]).
\end{align}
Let $\mathcal Y_D$ be the subset of $\mathcal Y$ consisting of all Dirac type measures: 
\begin{equation}
\label{eq:dirac-young}
  \mathcal Y_D := \left\{ \nu\in \mathcal Y: \nu_{t,x} = \delta_{\rho(t,x)} \text{ a.e., for some } \rho\in L^\infty(\Sigma_T),   0\le \rho(t,x) \le 1\right\}. 
\end{equation}
Recall the empirical density $\zeta_N$ in \eqref{eq:empirical} and define the Young measures 
\begin{align}
\label{eq:young}
  \nu^N := \big\{\nu^N_{t,x}=\delta_{\zeta_N(t,x)};\ (t,x)\in\Sigma_T\big\}. 
\end{align}
Since $\zeta_N \in [0,1]$ for all $N$, $\nu^N$ generates a probability measure on $\mathcal Y$ that concentrates on $\mathcal Y_D$. 
To prove \eqref{eq:convergence}, it suffices to show that $\nu^N$ converge, in probability, to the Dirac type measure $\{\delta_{\rho(t,x)};(t,x)\in\Sigma_T\}$ with the entropy solution $\rho$ to \eqref{eq:cl1}--\eqref{eq:cl2}. 

\begin{proof}[Proof of Theorem \ref{thm:hyd}]
Let $K=K(N)$ be a mesoscopic scale such that $K=o(N)$.
The choice of $K$ will be specified later in \eqref{eq:assp-k}.
Define the uniform average 
\begin{align}
\label{eq:block}
  \bar\eta_{i,K} := \frac1K\sum_{i'=0}^{K-1} \eta_{i-i'}, \quad \forall\,i=K,K+1,\ldots,N. 
\end{align}
For $i=K$, ..., $N-K+1$, define the smoothly weighted average
\begin{align}
\label{eq:smooth-block}
  \hat\eta_{i,K} := \frac1K\sum_{i'=0}^{K-1} \bar\eta_{i+i',K} = \sum_{i'=-K+1}^{K-1} w_{i'}\eta_{i-i'}, \quad w_{i'} = \frac{K-|i'|}{K^2}. 
\end{align}
The corresponding empirical density $\rho_N$ is given by 
\begin{align}
\label{eq:empirical1}
  \rho_N(t,x) := \sum_{i=K+1}^{N-K} \chi_{i,N}(x) \hat\eta_{i,K}(t), \quad (t,x) \in [0,+\infty)\times[0,1], 
\end{align}
where $\chi_{i,N}$ is the indicator function in \eqref{eq:indicator}. 
Observe that for any $\psi \in \cC^1(\Sigma_T)$, 
\begin{align}
  \lim_{N\to\infty} \left| \iint_{\Sigma_T} \psi(t,x)\big(\zeta_N(t,x)-\rho_N(t,x)\big)dx\,dt \right| \le \lim_{N\to\infty} \frac{C_\psi TK}N = 0. 
\end{align}
Therefore, $\rho_N$ can be identified with $\zeta_N$ in the macroscopic limit. 
In the following we prove Theorem \ref{thm:hyd} with $\eta_i$ replaced by $\hat\eta_{i,K}$.

Denote by $\hat\nu_N$ the Young measure concentrated on $\rho_N$. 
Let $\fQ_N$ be the probability distribution of $\hat\nu_N$ on $\mathcal Y$. 
Since $\mathcal Y$ is compact, the sequence $\fQ_N$ is tight, thus we can extract a subsequence $\fQ_{N'}$ that converges weakly to some measure $\fQ$ on $\mathcal Y$, which means that for all $f \in \cC(\Sigma_T\times[0,1])$, 
\begin{align*}
  E^\fQ \left[\iint_{\Sigma_T} dx\,dt\int_0^1 f(t,x,u)\nu_{t,x}(du)\right] 
  &= \lim_{N'\to\infty} E^{\fQ_{N'}} \left[\iint_{\Sigma_T} dx\,dt\int_0^1 f(t,x,u)\nu_{t,x}(du)\right] \\
  &= \lim_{N'\to\infty} \bE_{N'} \left[\iint_{\Sigma_T} f\big(t,x,\rho_{N'}(t,x)\big)dx\,dt\right]. 
\end{align*}

The proof is then divided into two main steps.
First, in Section \ref{sec:com-com} we show that there exists $\rho \in L^\infty(\Sigma_T)$, which may depend on the choice of subsequence $N'$, such that
\begin{align}
\label{eq:dirac-delta}
  \fQ\left\{\nu_{t,x}=\delta_{\rho(t,x)}, \ (t,x)-\text{a.s.}\right\}=1.
\end{align}
Second, in Section \ref{sec:bd-ent} we prove that $\rho=\rho(t,x)$ satisfies that
\begin{equation}
\label{eq:ent-sol}
  \begin{aligned}
  &-\iint_{\Sigma_T} F(\rho,h)\partial_t\psi\,dx\,dt - p\iint_{\Sigma_T} Q(\rho,h)\partial_x\psi\,dx\,dt\\
  \le\,&\int_0^1 F(u_0,h)\psi(0,\cdot)dx + p\int_0^T F(\rho_-,h)\psi(\cdot,0)dt + p\int_0^T F(\rho_+,h)\psi(\cdot,1)dt
  \end{aligned}
\end{equation}
for all convex boundary entropy flux pair $(F,Q)$, any $h\in\bR$ and smooth, positive function $\psi$ defined on $\Sigma_T$.
By \cite[Theorem 2.7.31]{MNRR96}, \eqref{eq:ent-sol} is a necessary and sufficient condition of \eqref{eq:ent-sol1}--\eqref{eq:ent-sol3}.
Hence, $\fQ$ is concentrated on the Dirac type Young measure associated to the entropy solution to \eqref{eq:cl1}--\eqref{eq:cl2}.
Finally, we conclude from the uniqueness of the entropy solution \cite[Theorem 2.7.28]{MNRR96} that $\fQ$ is independent of the choice of subsequence.
\end{proof}

\subsection{Mesoscopic scale}
Hereafter we fix $K$ to be such that
\begin{align}
\label{eq:assp-k}
  \lim_{N\to\infty} \frac{N\sigma_N}{K^3} = \lim_{N\to\infty} \frac{NK^2}{\sigma_N^3} = \lim_{N\to\infty} \frac{\sigma_N^2}{NK} = 0.
\end{align}
This choice of the mesoscopic scale is crucial for the proof, for instance in Lemma \ref{lem:esti-m}, \ref{lem:esti-s} and \ref{lem:esti-a-bd}.
In view of \eqref{eq:assp}, a simple example is to choose
\begin{align}
  \sigma_N = N^{\frac57+\kappa}, \quad K = \left[ N^{\frac47+\frac{17\kappa}{14}+\kappa^2} \right], \quad \kappa \in \left( 0,\frac27 \right),
\end{align}
In particular, $N^{4/7} \ll K \ll \sigma_N \ll N$.

\section{Microscopic entropy production}

Given an entropy pair $(f,q)$ and a Young measure $\nu\in\mathcal Y$, the entropy production of $\nu$ is a distribution $X^f(\nu) \in W^{-1,\infty}(\Sigma_T)$ defined by 
\begin{align}
\label{eq:ent-prod0}
  X^f(\nu) := \partial_t \langle f \rangle_{\nu_{t,x}} + p\partial_x\langle q \rangle_{\nu_{t,x}}, \quad (t,x)\in\Sigma_T. 
\end{align}
Applied on a test function $\psi\in\cC^1(\Sigma_T)$, this definition reads
\begin{align}
\label{eq:ent-prod1}
  X^f(\nu,\psi) = -\iint_{\Sigma_T} \langle f \rangle_{\nu_{t,x}}\partial_t\psi\,dx\,dt - p\iint_{\Sigma_T} \langle q \rangle_{\nu_{t,x}}\partial_x\psi\,dx\,dt.
\end{align}
To simplify the notations, for $i=1$, ..., $N$ we write 
\begin{align}
\label{eq:partition}
  \psi_i(t) := \psi \left( t,\frac iN - \frac1{2N} \right), \quad \bar\psi_i(t) := N\int_0^1 \psi(t,x) \chi_{i,N}(x)dx. 
\end{align}
For a sequence $\{a_i\}$, denote $\nabla a_i=a_{i+1}-a_i$, $\nabla^*a_i=a_{i-1}-a_i$ and
\begin{align}
\label{eq:laplacian}
  \Delta a_i=(-\nabla^*\nabla)a_i=a_{i+1}-2a_i+a_{i-1}.
\end{align}

For $i=K+1$, ... $N-K$, recall $\hat\eta_{i,K}$ defined in \eqref{eq:smooth-block}.
Observe that they are supported on $\{\eta_2, ..., \eta_{N-1}\}$, so $L_{\pm,t}$ do not contribute to their time evolution and
\begin{align}
\label{eq:current}
  L_{N,t} [\hat\eta_{i,K}] = p\nabla^*\hat J_{i,K}+\sigma\Delta\hat\eta_{i,K}, \quad \hat J_{i,K} = \sum_{i'=-K+1}^{K-1} w_{i'}J_{i-i',i-i'+1},
\end{align}
where $J_{i,i+1}=\eta_i(1-\eta_{i+1})$ gives the microscopic current.
In most of the following contents, we omit the subscript $K$ in $\hat\eta_{i,K}$, $\hat J_{i,K}$ and write $\hat\eta_i$, $\hat J_i$ for short.

\begin{lem}[Basic decomposition]
\label{lem:ent-prod-decom}
Recall the empirical density $\rho_N$ in \eqref{eq:empirical1} and the Dirac type Young measure $\hat\nu_N$ concentrated on it.
The entropy production decomposes as
\begin{equation}
\label{eq:ent-prod-decom}
  \begin{aligned}
    X^f(\hat\nu_N,\psi) =\;&\frac1N\sum_{i=K+1}^{N-K} f\big(\hat\eta_i(0)\big)\bar\psi_i(0)\\
    &+ \mathcal M_N(\psi) + \cA_N(\psi) + \cS_N(\psi) + \sum_{\ell=1,2,3} \mathcal E_{N,\ell}(\psi),
  \end{aligned}
\end{equation}
for all $\psi\in\cC^1(\Sigma_T)$ such that $\psi(T,\cdot)=0$.
In \eqref{eq:ent-prod-decom},
\begin{align}
\label{eq:martingale}
  \mathcal M_N(\psi) := -\int_0^T \frac1N\sum_{i=K+1}^{N-K} \bar\psi'_i(t)M_{i,K}^f(t)dt
\end{align}
where $M_{i,K}^f=M_{i,K}^f(t)$ is a family of martingales,
\begin{align}
\label{eq:def-a}
  \cA_N(\psi) &:= p\int_0^T \sum_{i=K+1}^{N-K} \bar\psi_i(t)f'\big(\hat\eta_i(t)\big)\nabla^* \left[ \hat J_i(t) - J\big(\hat\eta_i(t)\big) \right] dt,\\
\label{eq:def-s}
  \cS_N(\psi) &:= \sigma_N\int_0^T \sum_{i=K+1}^{N-K} \bar\psi_i(t)f'\big(\hat\eta_i(t)\big)\Delta\hat\eta_i(t)dt,
\end{align}
where $J(u)=u(1-u)$ and for $\ell=1$, $2$, $3$,
\begin{align}
\label{eq:def-e1}
  \mathcal E_{N,1}(\psi) &:= p\int_0^T \sum_{i=K+1}^{N-K} \big[\bar\psi_i\nabla^*q(\hat\eta_i) - q(\hat\eta_i)\nabla\psi_i\big]dt,\\
\label{eq:def-e2} 
  \mathcal E_{N,2}(\psi) &:= \int_0^T \sum_{i=K+1}^{N-K} \bar\psi_i(t) \left( \ep_{i,K}^{(1)} + p\ep_{i,K}^{(2)} \right) dt,\\
\label{eq:def-e3}
  \mathcal E_{N,3}(\psi) &:= -pq(0)\int_0^T \left( \int_0^{\frac{2K+1}{2N}} + \int_{1-\frac{2K-1}{2N}}^1 \right) \partial_x\psi\,dx\,dt,
\end{align}
where the small correction terms are given by
\begin{align}
\label{eq:corrections}
  \ep_{i,K}^{(1)} := L_{N,t}\big[f(\hat\eta_i)\big] - f'(\hat\eta_i)L_{N,t}[\hat\eta_i], \quad
  \ep_{i,K}^{(2)} := f'(\hat\eta_i)\nabla^*J(\hat\eta_i) - \nabla^*q(\hat\eta_i). 
\end{align}
\end{lem}

\begin{proof}
By the definition in \eqref{eq:empirical1},
\begin{align}
\label{eq:ent-prod-decom0}
  -\iint_{\Sigma_T} f(\rho_N)\partial_t\psi\,dx\,dt = -\int_0^T \frac1N\sum_{i=K+1}^{N-K} f(\hat\eta_i)\bar\psi'_i\,dt.
\end{align}
By Dynkin formula, $M_{i,K}^f=M_{i,K}^f(t)$ is a martingale, where
\begin{align}
  M_{i,K}^f(t) := f\big(\hat\eta_i(t)\big) - f\big(\hat\eta_i(0)\big) - N\int_0^t L_{N,s} \big[f(\hat\eta_i(s))\big]ds.
\end{align}
Its quadratic variance is given by
\begin{align}
\label{eq:martingale1}
  \big\langle M_{i,K}^f \big\rangle(t) = N\int_0^t \Big\{L_{N,s} \big[f(\hat\eta_i)^2\big]-2f(\hat\eta_i)L_{N,s} \big[f(\hat\eta_i)\big]\Big\}ds.
\end{align}
Recalling \eqref{eq:current} and \eqref{eq:corrections}, we can decompose as
\begin{equation}
  \begin{aligned}
    &L_{N,t} [f(\hat\eta_i)] = \ep_{i,K}^{(1)} + f'(\hat\eta_i)\big(p\nabla^*\hat J_i+\sigma_N\Delta\hat\eta_i\big)\\
    =\;&\ep_{i,K}^{(1)} + p \left\{ f'(\hat\eta_i)\nabla^* \left[ \hat J_i-J(\hat\eta_i) \right] + \ep_{i,K}^{(2)} + \nabla^*q(\hat\eta_i) \right\} + \sigma_Nf'(\hat\eta_i)\Delta\hat\eta_i.
  \end{aligned}
\end{equation}
Therefore, \eqref{eq:ent-prod-decom0} and integrate-by-parts formula yield that
\begin{equation}
\label{eq:ent-prod-decom1}
  \begin{aligned}
    -\iint_{\Sigma_T} f(\rho_N)\partial_t\psi\,dx\,dt =\;&\frac1N\sum_{i=K+1}^{N-K} f\big(\hat\eta_{i,K}(0)\big)\bar\psi_i(0) + \mathcal M_N(\psi) + \cA_N(\psi)\\
    &+ \cS_N(\psi) + p\int_0^T \sum_{i=K+1}^{N-K} \bar\psi_i\nabla^*q(\hat\eta_i)dt + \mathcal E_{N,2}(\psi).
  \end{aligned}
\end{equation}
On the other hand, elementary manipulation shows that 
\begin{align}
\label{eq:ent-prod-decom2}
  - p\iint_{\Sigma_T} q(\rho_N)\partial_x\psi\,dx\,dt = -p\int_0^T \sum_{i=K+1}^{N-K} q(\hat\eta_i)\nabla\psi_i\,dt + \mathcal E_{N,3}(\psi), 
\end{align}
The decomposition \eqref{eq:ent-prod-decom} then follows from \eqref{eq:ent-prod1}, \eqref{eq:ent-prod-decom1} and \eqref{eq:ent-prod-decom2}.
\end{proof}

\section{Compensated compactness}
\label{sec:com-com}

The aim of this section is to verify \eqref{eq:dirac-delta}, following the stochastic compensated compactness developed in \cite{Fritz04}.
We first prepare some useful notations.
Denote by $\cC_0^\infty(\Sigma_T)$ the class of smooth functions on $\Sigma_T$ with compact support included in the interior of $\Sigma_T$.
Let $H_0^1(\Sigma_T)$ be the completion of $\cC_0^\infty(\Sigma_T)$ under the norm
\begin{align}
  \|\vf\|_{H_0^1(\Sigma_T)}^2 := \iint_{\Sigma_T} \big[\vf^2+(\partial_t\vf)^2+(\partial_x\vf)^2\big]dx\,dt.
\end{align}
The dual space of $H_0^1(\Sigma_T)$ is denoted by $H^{-1}(\Sigma_T)$, equipped with the norm
\begin{align}
  \|X\|_{H^{-1}(\Sigma_T)} := \sup \left\{ |X(\vf)|,\ \vf\in\cC_0^\infty(\Sigma_T),\ \|\vf\|_{H_0^1(\Sigma_T)} \le 1 \right\}.
\end{align}
Let $\mathcal M(\Sigma_T)$ be the set of signed Radon measures on $\Sigma_T$.
Define
\begin{align}
  \|\mu\|_{\mathcal M(\Sigma_T)} := \sup \big\{|\mu(\vf)|,\ \vf\in\cC_0^\infty(\Sigma_T),\ \|\vf\|_{L^\infty(\Sigma_T)} \le 1\big\}.
\end{align}

Recall the empirical density $\rho_N$ in \eqref{eq:empirical1} and the Young measure $\hat\nu_N$ associated to it.
Since $\hat\nu_N$ is of Dirac type, $\fQ_N(\mathcal Y_D) =1$, where $\fQ_N$ is the distribution of $\hat\nu_N$ and $\mathcal Y_D$ is defined in \eqref{eq:dirac-young}.
Recall that $\fQ$ is a weak limit point of $\fQ_N$.
To show $\fQ(\mathcal Y_D)=1$, the main obstacle is that $\mathcal Y_D$ is not closed under the weak topology.
In the following, we construct a set $\mathcal G \subseteq \mathcal Y$ such that the closure of $\mathcal Y_D\cap\mathcal G$ is contained in $\mathcal Y_D$ and $\lim_{N\to\infty} \fQ_N(\mathcal Y_D\cap\mathcal G) = 1$.

For two given Lax entropy flux pairs $(f_1,q_1)$, $(f_2,q_2)$ and two sequences $\{a_N\}$, $\{b_N\}$ such that $\lim_{N\to\infty} a_N=0$, $\sup b_N<\infty$, define
\begin{equation}
  \mathcal G_N := \left\{ \nu\in\mathcal Y_D \left|
  \begin{aligned}
    &\,\text{for}\ j=1,2,\ X^{f_j}(\nu) = Y_j + Z_j, \text{ s.t.}\\
    &\|Y_j\|_{H^{-1}(\Sigma_T)} \le a_N,\ \|Z_j\|_{\mathcal M(\Sigma_T)} \le b_N
  \end{aligned}
  \right.\right\},
\end{equation}
where $X^{f_j}(\nu)$ is the entropy production defined in \eqref{eq:ent-prod0}.
Let $\mathcal G$ be the closure of $\cap_{N\ge1}\mathcal G_N$ under the weak topology.
By the div-curl lemma  (\cite[Theorem 5.B.4, p.54]{Evans90}), $\mathcal G$ is a subset of $\mathcal Y_D$.
In Proposition \ref{prop:com-com} below we prove that $\lim_{N\to\infty} \fQ_N(\mathcal G_N) = 1$ and consequently that $\fQ(\mathcal Y_D) \ge \fQ(\mathcal G) = 1$ for any weak limit point $\fQ$.

\begin{prop}
\label{prop:com-com}
The entropy production $X^f(\hat\nu_N) = Y_N+Z_N$, such that
\begin{equation}
\label{eq:com-com}
  \begin{aligned}
    &\bE_N \big[\|Y_N\|_{H^{-1}(\Sigma_T)}\big] \le a_N, \quad \lim_{N\to\infty} a_N = 0, \\
    &\bE_N \big[\|Z_N\|_{\mathcal M(\Sigma_T)}\big] \le b_N, \quad \sup b_N < \infty. 
  \end{aligned}
\end{equation}
\end{prop}

In order to prove Proposition \ref{prop:com-com}, we apply \eqref{eq:ent-prod-decom} with $\vf\in\cC_0^\infty(\Sigma_T)$ and estimate each term in the right-hand side by several lemmas.
Recall the definition of $\vf_i$ and $\bar\vf_i$ in \eqref{eq:partition}.

\begin{lem}
\label{lem:esti-e}
As $N\to\infty$, $\mathcal E_{N,1}$ and $\mathcal E_{N,3}$ vanish uniformly in $H^{-1}(\Sigma_T)$, while $\mathcal E_{N,2}$ vanishes uniformly in $\mathcal M(\Sigma_T)$.
\end{lem}

\begin{proof}
We first treat $\mathcal E_{N,1}$.
Applying summation by parts\footnote{Though $\vf$ is compactly supported, the boundary terms can not be omitted autonomously, since we need to take supreme over all $\vf$ before send $N$ to $\infty$.},
\begin{equation}
  \begin{aligned}
    \mathcal E_{N,1}(\vf) = p\int_0^T \vf_{K+1}q(\hat\eta_K)dt - p\int_0^T \vf_{N-K+1}q(\hat\eta_{N-K})dt\\
    + p\int_0^T \sum_{i=K+1}^{N-K} (\bar\vf_i-\vf_i)\nabla^*q(\hat\eta_i)dt.
  \end{aligned}
\end{equation}
Since $\vf$ vanishes at the boundary of $\Sigma_T$, \CS inequality yields that
\begin{equation}
\label{eq:esti-bd}
  \begin{aligned}
    \int_0^T \big[\vf(t,x)^2+\vf(t,1-x)^2\big]dt &\le \int_0^T x \left[ \int_0^x (\partial_y\vf)^2dy+\int_{1-x}^1 (\partial_y\vf)^2dy \right] dt\\
    &\le 2x\|\vf\|_{H_0^1(\Sigma_T)}^2, \quad \forall x\in[0,1].
  \end{aligned}
\end{equation}
Consequently, $\int_0^T \vf_{K+1}^2dt + \int_0^T \vf_{N-K+1}^2dt \le CKN^{-1}\|\vf\|_{H_0^1}^2$ and
\begin{equation}
\label{eq:e1}
  \begin{aligned}
    &\left| \int_0^T \vf_{K+1}q(\hat\eta_K)dt - \int_0^T \vf_{N-K+1}q(\hat\eta_{N-K})dt \right|\\
    \le\;&C\sqrt{\frac KN}\|\vf\|_{H_0^1(\Sigma_T)} \left\{ \bigg[\int_0^T q(\hat\eta_K)^2dt\bigg]^{\frac12} + \bigg[\int_0^T q(\hat\eta_{N-K})^2dt\bigg]^{\frac12} \right\}\\
    \le\;&C'T|q|_\infty\sqrt{\frac KN}\|\vf\|_{H_0^1(\Sigma_T)}.
  \end{aligned}
\end{equation}
Also noting that $|\nabla^*q(\hat\eta_i)| \le |q'|_\infty|\hat\eta_{i-1}-\hat\eta_i| \le |q'|_\infty K^{-1}$,
\begin{equation}
  \begin{aligned}
    \bigg|\int_0^T \sum_{i=K+1}^{N-K} (\bar\vf_i-\vf_i)\nabla^*q(\hat\eta_i)dt\bigg|^2 &\le |q'|_\infty^2\frac{NT}{K^2}\int_0^T \sum_{i=K+1}^{N-K} (\bar\vf_i-\vf_i)^2dt\\
  &\le C|q'|_\infty^2 K^{-2} \|\vf\|_{H_0^1(\Sigma_T)}^2.
  \end{aligned}
\end{equation}
For $\mathcal E_{N,3}$, by \CS inequality,
\begin{align}
  \big|\mathcal E_{N,3}(\vf)\big|^2 \le \frac{2p^2|q(0)|^2TK}N \iint_{\Sigma_T} (\partial_x\vf)^2dx\,dt \le \frac{C|q|_\infty^2TK}N\|\vf\|_{H_0^1(\Sigma_T)}^2.
\end{align}
The conclusion then follows from \eqref{eq:assp} and \eqref{eq:assp-k}.

Now we estimate $\mathcal E_{N,2}$ in $\mathcal M(\Sigma_T)$.
By the definition \eqref{eq:gen} of $L_{N,t}$, 
\begin{align}
  \ep_{i,K}^{(1)} = \sum_{j =i-K}^{i+K-1} (p\eta_j+\sigma_N) \left[ f(\hat\eta_i^{j,j+1}) - f(\hat\eta_i) - f'(\hat\eta_i) \left( \hat\eta_i^{j,j+1} - \hat\eta_i \right) \right].
\end{align}
Recall $\hat\eta_i=\hat\eta_{i,K}$ in \eqref{eq:smooth-block}.
Direct computation shows that
\begin{equation}
\label{eq:exchange-smooth}
  \hat\eta_i^{j,j+1} = \hat\eta_i - \sgn\left( i-j-\frac12 \right)\frac{\nabla\eta_j}{K^2}, \quad j-K+1 \le i \le j+K 
\end{equation}
and otherwise $\hat\eta_i^{j,j+1}-\hat\eta_i=0$.
Therefore,
\begin{align}
\label{eq:esti-ep1}
  \left| \ep_{i,K}^{(1)} \right| \le \sum_{j=i-K}^{i+K-1} \frac{(p+\sigma_N)|f''|_\infty}2 \left( \hat\eta_i^{j,j+1} - \hat\eta_i \right)^2 \le \frac{(p+\sigma_N)|f''|_\infty}{K^3}.
\end{align}
Also notice that $f'(u)J'(u)=q'(u)$ and $|\nabla^*\hat\eta_{i,K}| \le 1/K$, so that 
\begin{align}
\label{eq:esti-ep2}
  \left| \ep_{i,K}^{(2)} \right| \le \frac{|f'|_\infty|J''|_\infty+|q''|_\infty}{K^2} \le \frac{4|f'|_\infty+|f''|_\infty}{K^2}. 
\end{align}
From \eqref{eq:esti-ep1} and \eqref{eq:esti-ep2}, 
\begin{align}
\label{eq:esti-e2}
  \big|\mathcal E_{N,2}(\vf)\big| \le \left[ \frac{pN\sigma_N|f''|_\infty}{K^3}+\frac{p(4|f'|_\infty+|f''|_\infty)N}{K^2} \right] \|\vf\|_{L^\infty(\Sigma_T)}. 
\end{align}
In view of \eqref{eq:assp-k}, $\mathcal E_{N,2}$ is negligible in $\mathcal M(\Sigma_T)$.
\end{proof}

\begin{lem}
\label{lem:esti-m}
As $N\to\infty$, $\bE_N [\|\mathcal M_N\|_{H^{-1}(\Sigma_T)}]$ vanishes.
\end{lem}

\begin{proof}
From \eqref{eq:martingale} we see that
\begin{equation}
  \begin{aligned}
    \big|\mathcal M_N(\vf)\big|^2 &\le \int_0^T \frac1N\sum_{i=K+1}^{N-K} \big|M_{i,K}^f(t)\big|^2dt \int_0^T \frac1N\sum_{i=K+1}^{N-K} \bar\vf'(t)^2dt\\
    &\le \frac1N\sum_{i=K+1}^{N-K} \int_0^T \big|M_{i,K}^f(t)\big|^2dt \times C\|\vf\|_{H_0^1(\Sigma_T)}^2.
  \end{aligned}
\end{equation}
The definition of $\|\cdot\|_{H^{-1}(\Sigma_T)}$ together with Doob's inequality yields that
\begin{equation}
  \begin{aligned}
    \bE_N \left[ \big\|\mathcal M_N\big\|_{H^{-1}(\Sigma_T)}^2 \right] &\le \frac CN\sum_{i=K+1}^{N-K} \bE_N \left[ \int_0^T \big|M_{i,K}^f(t)\big|^2dt \right]\\
    &\le \frac{C'}N\sum_{i=K+1}^{N-K} \bE_N \left[ \int_0^T \big\langle M_{i,K}^f \big\rangle(t)dt \right].
  \end{aligned}
\end{equation}
From \eqref{eq:martingale1}, the quadratic variance reads
\begin{equation}
  \begin{aligned}
    \bE_N \left[ \big\langle M_{i,K}^f \big\rangle(t) \right] &= N\int_0^t \sum_{j=1}^{N-1} (p\eta_j+\sigma_N) \left[ f(\hat\eta_i^{j,j+1}) - f(\hat\eta_i) \right]^2ds\\
    &\le CN\sigma_N|f'|_\infty^2\int_0^t \sum_{j=1}^{N-1} \left( \hat\eta_i^{j,j+1} - \hat\eta_i \right)^2dt \le \frac{C|f'|_\infty^2N\sigma_N}{K^3},
  \end{aligned}
\end{equation}
where the last inequality follows from \eqref{eq:exchange-smooth}.
Therefore,
\begin{align}
  \bE_N \left[ \big\|\mathcal M_N\big\|_{H^{-1}(\Sigma_T)}^2 \right] \le \frac{C'|f'|_\infty^2N\sigma_N}{K^3}. 
\end{align}
The right-hand side vanishes from \eqref{eq:assp-k}.
\end{proof}

To decompose $\cA_N$ and $\cS_N$, we make use of the following block estimates. 
Their proofs are postponed to Section \ref{subsec:block}. 

\begin{prop}[One-block estimate]
\label{prop:one-block}
There is some constant $C$, such that 
\begin{align}
  \bE_N \left[ \int_0^T \sum_{i=K}^{N-K} \left[ \hat J_i-J(\hat\eta_i) \right]^2dt \right] \le C \left( \frac{K^2}{\sigma_N} + \frac NK \right). 
\end{align}
\end{prop}

\begin{prop}[$H^1$ estimate]
\label{prop:h1}
There is some constant $C$, such that 
\begin{align}
  \bE_N \left[ \int_0^T \sum_{i=K}^{N-K} (\nabla\hat\eta_i)^2dt \right] \le C \left( \frac1{\sigma_N} + \frac N{K^3} \right). 
\end{align}
\end{prop}

\begin{lem}
\label{lem:esti-a}
$\cA_N$ satisfies the decomposition in \eqref{eq:com-com}. 
\end{lem}

\begin{proof}
Without loss of generality, we fix $p=1$.
With $g_i:=\hat J_i-J(\hat\eta_i)$,
\begin{equation}
\label{eq:a1}
  \begin{aligned}
    \cA_N(\vf) = \int_0^T \sum_{i=K+1}^{N-K} f'(\hat\eta_{i+1})g_i\nabla\bar\vf_i\,dt + \int_0^T \sum_{i=K+1}^{N-K} \bar\vf_ig_i\nabla f'(\hat\eta_i)dt\\
    + \int_0^T \big[\bar\vf_{K+1}f'(\hat\eta_{K+1})g_K - \bar\vf_{N-K+1}f'(\hat\eta_{N-K+1})g_{N-K}\big]\,dt.
  \end{aligned}
\end{equation}
Similarly to \eqref{eq:e1}, the boundary integrals vanish uniformly in $H^{-1}(\Sigma_T)$ as $|g_i|\le1$.

Denote the first two integrals in \eqref{eq:a1} by $\cA_{N,1}(\vf)$ and $\cA_{N,2}(\vf)$, respectively.
Applying \CS inequality,
\begin{equation}
  \begin{aligned}
    \big|\cA_{N,1}(\vf)\big|^2 &\le |f'|_\infty^2 \int_0^T \sum_{i=K+1}^{N-K} g_i^2dt \int_0^T \sum_{i=K+1}^{N-K} (\nabla\bar\vf_i)^2dt\\
    &\le \frac{C|f'|_\infty^2}N\|\vf\|_{H_0^1(\Sigma_T)}^2 \int_0^T \sum_{i=K+1}^{N-K} g_i^2dt.
  \end{aligned}
\end{equation}
Proposition \ref{prop:one-block} then shows that
\begin{equation}
  \begin{aligned}
  \bE_N \left[ \big\|\cA_{N,1}\big\|_{H^{-1}(\Sigma_T)}^2 \right]
  &\le \frac{C|f'|_\infty^2}N \bE_N \left[ \int_0^T \sum_{i=K+1}^{N-K} g_i^2dt \right]\\
  &\le C'|f'|_\infty^2 \left( \frac{K^2}{N\sigma_N}+\frac 1K \right).
  \end{aligned}
\end{equation}
For the term $\cA_{N,2}$, again by \CS inequality,
\begin{equation}
  \begin{aligned}
    \big|\cA_{N,2}(\vf)\big|^2
    &\le |f''|_\infty^2\|\vf\|_{L^\infty(\Sigma_T)}^2 \bigg(\int_0^T \sum_{i=K+1}^{N-K} \big|g_i\nabla\hat\eta_i\big|dt\bigg)^2\\
    &\le |f''|_\infty^2\|\vf\|_{L^\infty(\Sigma_T)}^2 \int_0^T \sum_{i=K+1}^{N-K} g_i^2dt \int_0^T \sum_{i=K+1}^{N-K} \big(\nabla\hat\eta_i\big)^2dt. 
  \end{aligned}
\end{equation}
Thanks to Proposition \ref{prop:one-block} and \ref{prop:h1}, we obtain the estimate
\begin{equation}
\label{eq:a2}
  \begin{aligned}
    \bE_N \left[ \big\|\cA_{N,2}\big\|_{\mathcal M(\Sigma_T)} \right]
    &\le C|f''|_\infty \sqrt{\frac{K^2}{\sigma_N}+\frac NK}\sqrt{\frac1{\sigma_N}+\frac N{K^3}}\\
    &= C|f''|_\infty \left( \frac K{\sigma_N}+\frac N{K^2} \right).
  \end{aligned}
\end{equation}
Finally, we conclude the result from \eqref{eq:assp} and \eqref{eq:assp-k}.
\end{proof}

\begin{lem}
\label{lem:esti-s}
$\cS_N$ satisfies the decomposition in \eqref{eq:com-com}.
\end{lem}

\begin{proof}
Recall that $\Delta=-\nabla^*\nabla$.
Similarly to the previous proof,
\begin{equation}
\label{eq:s1}
  \begin{aligned}
    \cS_N(\vf) = \cS_{N,1}(\vf) + \cS_{N,2}(\vf) - \sigma_N\int_0^T \bar\vf_{K+1}f'\big(\hat\eta_{K+1}\big)\nabla\hat\eta_K\,dt\\
    +\;\sigma_N\int_0^T \bar\vf_{N-K+1}f'\big(\hat\eta_{N-K+1}\big)\nabla\hat\eta_{N-K}\,dt,
  \end{aligned}
\end{equation}
where
\begin{equation}
\label{eq:s2}
  \begin{aligned}
  \cS_{N,1}(\vf) &= -\sigma_N\int_0^T \sum_{i=K+1}^{N-K} f'(\hat\eta_{i+1})\nabla\hat\eta_i\nabla\bar\vf_i\,dt,\\
  \cS_{N,2}(\vf) &= -\sigma_N\int_0^T \sum_{i=K+1}^{N-K} \bar\vf_i\nabla\hat\eta_i\nabla f'(\hat\eta_i)dt.
  \end{aligned}
\end{equation}
For the boundary integrals, noting that $|\nabla\hat\eta_i| \le K^{-1}$, by \eqref{eq:esti-bd} we have
\begin{align}
  \bE_N \left[ \sigma_N\int_0^T \bar\vf_{K+1}f'(\hat\eta_{K+1})\nabla\hat\eta_K\,dt \right] \le T|f'|_\infty\frac{\sigma_N}{\sqrt{NK}}\|\vf\|_{H_0^1}.
\end{align}
The term with $\bar\vf_{N-K+1}$ is estimated in the same way.

To deal with $\cS_{N,1}$, first apply \CS inequality to obtain
\begin{equation}
  \begin{aligned}
    \big|\cS_{N,1}(\vf)\big|^2
    &\le |f'|_\infty^2\sigma_N^2 \int_0^T \sum_{i=K+1}^{N-K} \big(\nabla\hat\eta_i\big)^2dt \int_0^T \sum_{i=K+1}^{N-K} (\nabla\bar\vf_i)^2dt\\
    &\le \frac{C|f'|_\infty^2\sigma_N^2}N\|\vf\|_{H_0^1(\Sigma_T)}^2 \int_0^T \sum_{i=K+1}^{N-K} \big(\nabla\hat\eta_i\big)^2dt.
  \end{aligned}
\end{equation}
Therefore, using Proposition \ref{prop:h1} we get
\begin{align}
  \bE_N \left[ \big\|\cS_{N,1}\big\|_{H^{-1}(\Sigma_T)}^2 \right] \le C|f'|_\infty^2 \left( \frac{\sigma_N}N+\frac{\sigma_N^2}{K^3} \right).
\end{align}

The term $\cS_{N,2}$ is bounded in $\mathcal M(\Sigma_T)$.
Indeed,
\begin{align}
  \big|\cS_{N,2}(\vf)\big| \le |f''|_\infty\sigma_N\|\vf\|_{L^\infty(\Sigma_T)} \int_0^T \sum_{i=K+1}^{N-K} (\nabla\hat\eta_i)^2dt.
\end{align}
Therefore, Proposition \ref{prop:h1} gives the estimate
\begin{align}
  \bE_N \left[ \big\|\cS_{N,2}\big\|_{\mathcal M(\Sigma_T)} \right] \le C|f''|_\infty \left( 1+\frac{N\sigma_N}{K^3} \right). 
\end{align}
We again conclude the result from \eqref{eq:assp} and \eqref{eq:assp-k}. 
\end{proof}

\begin{rem}
It is clear from the above proof that $\cS_{N,2}$ in Lemma \ref{lem:esti-s} is the only term which survives in the limit $N\to\infty$.
It is the microscopic origin of the non-zero macroscopic entropy production appeared in \eqref{eq:ent-sol1}.
\end{rem}

\section{Boundary entropy production}
\label{sec:bd-ent}

From Proposition \ref{prop:com-com}, any weak-$\star$ limit point of $\rho_N$ concentrates on some $\rho \in L^\infty(\Sigma_T)$.
Note that $\rho$ may depend on the choice of subsequence.
In this section we show that $\rho$ satisfies \eqref{eq:ent-sol}, thus is the unique $L^\infty$ entropy solution of \eqref{eq:cl1}--\eqref{eq:cl2}.

Before stating the result of this section, we point out the main difficulty that prevents us from applying the direct approach in Section \ref{sec:com-com} and \cite{DMOX22} here.
Taking a test function $\psi$ that is not compactly supported, the estimate \eqref{eq:esti-bd} does not hold any more.
Hence, in performing summation by parts in $\cS_N(\psi)$ like in Lemma \ref{lem:esti-s}, the boundary integrals are out of control because of the large factor $\sigma_N$.
In this section, we introduce a grading argument to treat the divergent integrals mentioned above.
The following definition of the auxiliary functions plays a central role.
For each $N$, define for $x\in[0,1]$ that
\begin{align}
\label{eq:auxiliary}
  \alpha_N(x) := 1 - \mathbf1_{x\in[0,\delta_N]} \Big(\frac{\sigma_N}{\sigma_N+1}\Big)^{Nx} - \mathbf1_{x\in(\delta_N,1]}\Big( \frac{\sigma_N-1}{\sigma_N}\Big)^{N(1-x)},
\end{align}
where $\delta_N\in(0,1)$ is chosen as
\begin{align}
  \delta_N = \frac{\log\sigma_N-\log(\sigma_N-1)}{\log(\sigma_N+1)-\log(\sigma_N-1)}. 
\end{align}
Observe that $\delta_N \approx 2^{-1}+(4\sigma_N)^{-1}+o(\sigma_N^{-2})$ for large $N$. 
For sufficiently large $N$, the function $\alpha_N$ is continuous, piecewise smooth and satisfies that 
\begin{itemize}
\item[(\romannum1).] $\alpha_N(0)=\alpha_N(1)=0$, $\alpha_N(x)\in[0,1)$ for $x\in[0,1]$; 
\item[(\romannum2).] $\lim_{N\to\infty}\alpha_N(x)=1$ for $x\in(0,1)$, uniformly on any compact subset; 
\item[(\romannum3).] $|\alpha'_N(x)| \le 2N\sigma_N^{-1}$ for $x\in(0,\delta_N)\cup(\delta_N,1)$ and $\int |\alpha'_N|dx \le 2$.  
\end{itemize}
By (\romannum2) together with the boundedness of $f(\rho_N)$ and $q(\rho_N)$, to show \eqref{eq:ent-sol} it suffices to prove the next proposition and send $N\to\infty$. 

\begin{prop}
\label{prop:bd-ent-prod}
Assume \eqref{eq:assp} and \eqref{eq:assp-k}, then 
\begin{equation}
  \begin{aligned}
    \lim_{N\to\infty} \bP_N \bigg\{ &-\iint_{\Sigma_T} F(\rho_N,h)\alpha_N\partial_t\psi\,dx\,dt\\
    &- p\iint_{\Sigma_T} Q(\rho_N,h)\alpha_N\partial_x\psi\,dx\,dt \le \Theta_{u_0,\rho_\pm}^{F,h}(\psi) \bigg\}=1,
  \end{aligned}
\end{equation}
for all boundary entropy flux pair $(F,Q)$, $h\in\bR$ and $\psi\in\cC^\infty(\Sigma_T)$ such that $\psi\ge0$, $\psi(T,\cdot)=0$, where $u_0$, $\rho_\pm$ are the initial and boundary data in \eqref{eq:cl2}, and
\begin{align}
  \Theta_{u_0,\rho_\pm}^{F,h}(\psi) := \int_0^1 f(u_0)\psi(0,\cdot)dx + p\int_0^T \big[f(\rho_-)\psi(\cdot,0) + f(\rho_+)\psi(\cdot,1)\big]dt.
\end{align}
\end{prop}

In the following proof, we fix $p=1$, an arbitrary convex boundary entropy flux pair $(F,Q)$ and $h\in\bR$.
We also adopt the short notations $\hat\eta_i$, $\hat J_i$ for $\hat\eta_{i,K}$, $\hat J_{i,K}$.
Denote $(f,q) = (F,Q)(\cdot,h)$, then $(f,q)$ is an entropy flux pair such that $f\ge0$, $f(h)=0$, $f''\ge0$ and
\begin{align}
\label{eq:esti-ent}
  |q(u)| = \left| \int_h^u q'(v)dv \right| \le |J'|_\infty\int_h^u f'(v)dv \le f(u), \quad \forall u\in\bR.
\end{align}
For $\psi$ in Proposition \ref{prop:bd-ent-prod}, recall $\psi_i$, $\bar\psi_i$ defined in \eqref{eq:partition}.
Furthermore, denote 
\begin{align*}
  &\psi^N(t,x) := \psi(t,x)\alpha_N(x), \quad \alpha_i^N := \alpha_N \left( \frac iN-\frac1{2N} \right), \quad \psi_i^N(t) := \psi_i(t)\alpha_i^N, \\
  &\bar\psi_i^N(t) := N\int_0^1 \psi^N(t,x)\chi_{i,N}(x)dx = N\int_{\frac iN-\frac1{2N}}^{\frac iN+\frac1{2N}} \psi(t,x)\alpha_N(x)dx. 
\end{align*}
Recall the operators $\nabla$, $\nabla^*$ and $\Delta$ defined in \eqref{eq:laplacian}.
Observe that $|\nabla\alpha_i^N| \le \sigma_N^{-1}$ and $|\nabla\psi_i^N| \le C\sigma_N^{-1}$ for all $i$. 

\begin{proof}[Proof of Proposition \ref{prop:bd-ent-prod}]
Applying Lemma \ref{lem:ent-prod-decom} with $\psi=\psi^N$,
\begin{equation}
  \begin{aligned}
    X^f(\hat\nu_N,\psi^N) =\;&\frac1N\sum_{i=K+1}^{N-K} f\big(\hat\eta(0)\big)\bar\psi_i^N(0) + \mathcal M_N(\psi^N)\\
    &+ \cA_N(\psi^N) + \cS_N(\psi^N) + \sum_{\ell=1,2,3} \mathcal E_{N,\ell}(\psi^N).
  \end{aligned}
\end{equation}
Since $\partial_x\psi^N=\alpha_N\partial_x\psi+\psi\alpha'_N$, from \eqref{eq:def-e3} we have
\begin{align}
  \mathcal E_{N,3}(\psi^N) = \mathcal E_{N,3}^*(\psi,\alpha_N) - q(0)\int_0^T \left( \int_0^{\frac{2K+1}{2N}} + \int_{1-\frac{2K-1}{2N}}^1 \right) \psi\alpha'_N\,dx\,dt,
\end{align}
where
\begin{align}
\label{eq:def-e3'}
  \mathcal E_{N,3}^*(\psi,\alpha_N) &:= -q(0)\int_0^T \left( \int_0^{\frac{2K+1}{2N}} + \int_{1-\frac{2K-1}{2N}}^1 \right) \alpha_N\partial_x\psi\,dx\,dt.
\end{align}
On the other hand, by the definition of $X^f(\nu,\psi)$ in \eqref{eq:ent-prod1},
\begin{equation}
  \begin{aligned}
    X^f(\hat\nu_N,\psi^N) =\,&-\iint_{\Sigma_T} \big[f(\rho_N)\partial_t\psi+q(\rho_N)\partial_x\psi\big]\alpha_N\,dx\,dt\\
    &- \iint_{\Sigma_T} q(\rho_N)\psi\alpha'_N\,dx\,dt.
  \end{aligned}
\end{equation}
We furthermore compute the last term above as
\begin{equation}
  \begin{aligned}
    \iint_{\Sigma_N} q(\rho_N)\psi\alpha'_N\,dx\,dt =\,&\int_0^T \sum_{i=K+1}^{N-K} q\big(\hat\eta_i\big)\int_{\frac iN-\frac1{2N}}^{\frac iN+\frac1{2N}} \psi\alpha'_N\,dx\,dt\\
    &+ q(0)\int_0^T \left( \int_0^{\frac{2K+1}{2N}} + \int_{1-\frac{2K-1}{2N}}^1 \right) \psi\alpha'_N\,dx\,dt.
  \end{aligned}
\end{equation}
Therefore,
\begin{equation}
  \begin{aligned}
    &-\iint_{\Sigma_T} \big[f(\rho_N)\partial_t\psi + q(\rho_N)\partial_x\psi\big]\alpha_N\,dx\,dt\\
    =\;&\frac1N\sum_{i=K+1}^{N-K} f\big(\hat\eta(0)\big)\bar\psi_i^N(0) + \mathcal M_N(\psi^N) + \cA_N(\psi^N) + \cS_N(\psi^N)\\
    & + \sum_{\ell=1,2} \mathcal E_{N,\ell}(\psi^N) + \mathcal E_{N,3}^*(\psi,\alpha_N) + \int_0^T \sum_{i=K+1}^{N-K} q\big(\hat\eta_i\big) \int_{\frac iN-\frac1{2N}}^{\frac iN+\frac1{2N}} \psi\alpha'_N\,dx\,dt. 
  \end{aligned}
\end{equation}
Rewrite the last integral above as $\mathcal B_N(\psi,\alpha_N)+\mathcal E_{N,4}^*(\psi,\alpha_N)$, where
\begin{align}
  \label{eq:def-b}
  \mathcal B_N(\psi,\alpha_N) &:= \int_0^T \sum_{i=K+1}^{N-K} q\big(\hat\eta_i\big)\psi_i\nabla\alpha_i^Ndt, \\
  \label{eq:def-e4'}
  \mathcal E_{N,4}^*(\psi,\alpha_N) &:= \int_0^T \sum_{i=K+1}^{N-K} q\big(\hat\eta_i\big) \int_{\frac iN-\frac1{2N}}^{\frac iN+\frac1{2N}} (\psi-\psi_i)\alpha'_N\,dx\,dt. 
\end{align}
We then obtain the decomposition
\begin{equation}
\label{eq:bd-ent-prod-decom}
  \begin{aligned}
    &-\iint_{\Sigma_T} \big[f(\rho_N)\partial_t\psi + q(\rho_N)\partial_x\psi\big]\alpha_N\,dx\,dt\\
    =\;&\frac1N\sum_{i=K+1}^{N-K} f\big(\hat\eta(0)\big)\bar\psi_i^N(0) + \mathcal M_N(\psi^N) + \cA_N(\psi^N) + \cS_N(\psi^N)\\
    &+ \mathcal B_N(\psi,\alpha_N) + \sum_{\ell=1,2} \mathcal E_{N,\ell}(\psi^N) + \sum_{\ell=3,4} \mathcal E_{N,\ell}^*(\psi,\alpha_N).
  \end{aligned}
\end{equation}
Taking the limit $N\to\infty$ in \eqref{eq:bd-ent-prod-decom}, Proposition \ref{prop:bd-ent-prod} then follows from \eqref{eq:assp-initial} and Lemma \ref{lem:esti-e-bd}--\ref{lem:esti-b} established below.
\end{proof}

In the rest of this section, we fix $\psi$ and estimate the limit of each term in the right-hand side of \eqref{eq:bd-ent-prod-decom}.
We begin with terms that vanish uniformly.

\begin{lem}
\label{lem:esti-e-bd}
As $N\to\infty$, $|\mathcal E_{N,1}(\psi^N)|$, $|\mathcal E_{N,2}(\psi^N)|$, $|\mathcal E_{N,3}^*(\psi,\alpha_N)|$ and $|\mathcal E_{N,4}^*(\psi,\alpha_N)|$ converges to $0$ uniformly.
\end{lem}

\begin{proof}
Recall that the definition of $\mathcal E_{N,1}(\psi^N)$ reads
\begin{equation}
\label{eq:e-bd1}
  \begin{aligned}
    \mathcal E_{N,1}(\psi^N) = \,&\int_0^T \sum_{i=K+1}^{N-K} \big[\bar\psi_i^N\nabla^*q(\hat\eta_i) - q(\hat\eta_i)\nabla\psi_i^N\big] dt\\
    =\,&\int_0^T \psi_{K+1}^Nq(\hat\eta_K)dt - \int_0^T \psi_{N-K+1}^Nq(\hat\eta_{N-K})dt\\
  &+ \int_0^T \sum_{i=K+1}^{N-K} \big(\bar\psi_i^N-\psi_i^N\big)\nabla^*q(\hat\eta_i)dt.
  \end{aligned}
\end{equation}
Since $K \ll \sigma_N \ll N$, for some $c>0$ and sufficiently large $N$,
\begin{align}
\label{eq:bd-estimate}
  \alpha_{K+1}^N = 1- \big[1-(1+\sigma_N)^{-1}\big]^{K+\frac12} \le 1-e^{-\frac{(1+c)K}{\sigma_N}} \le \frac{(1+c)K}{\sigma_N}.
\end{align}
As $\psi_{K+1}^N=\psi_{K+1}\alpha_{K+1}^N$, the first integral in \eqref{eq:e-bd1} is vanishing.
The second one follows similarly.
To deal with the last term in \eqref{eq:e-bd1}, note that for each $i$,
\begin{equation}
  \begin{aligned}
    \left| \bar\psi_i^N-\psi_i^N \right|
    &\le N\int_{\frac iN-\frac1{2N}}^{\frac iN+\frac1{2N}} \left| \psi(t,x)\alpha_N(x) - \psi_i\alpha_i^N \right| dx\\
    &\le \frac1N\big(|\alpha_N|_\infty\|\partial_x\psi\|_{L^\infty(\Sigma_T)} + \|\psi\|_{L^\infty}|\alpha'_N|_\infty\big) \le \frac C{\sigma_N},
  \end{aligned}
\end{equation}
where the last inequality follows from (\romannum3).
As $|\nabla^*q(\hat\eta_i)| \le C'K^{-1}$,
\begin{align}
  \left| \int_0^T \sum_{i=K+1}^{N-K} \big(\bar\psi_i^N-\psi_i^N\big)\nabla^*q(\hat\eta_i)dt \right| \le \frac{C''N}{\sigma_NK}. 
\end{align}
In view of \eqref{eq:assp} and \eqref{eq:assp-k}, $\mathcal E_{N,1}(\psi^N)$ vanishes when $N\to\infty$.

For $|\mathcal E_{N,2}(\psi^N)|$ and $|\mathcal E_{N,3}^*(\psi,\alpha_N)|$, observe that they are respectively bounded from above by $|\mathcal E_{N,2}(\psi)|$ and $|\mathcal E_{N,3}(\psi)|$.
The conclusion then follows directly from Lemma \ref{lem:esti-e}.

We are left with $\mathcal E_{N,4}^*(\psi,\alpha_N)$.
The estimate is also straightforward:
\begin{align}
  \big|\mathcal E_{N,4}^*(\psi,\alpha_N)\big| \le 
|q|_\infty\iint_{\Sigma_T} \big|(\psi-\psi_i)\alpha'_N\big|\,dx \le \frac CN\int_0^1 |\alpha'_N|\,dx.
\end{align}
Noting that $\int_{[0,1]} |\alpha'_N|dx\equiv2$, $\mathcal E_{N,4}^*(\psi,\alpha_N)$ also vanishes.
\end{proof}

\begin{lem}
\label{lem:esti-m-bd}
As $N\to\infty$, $|\mathcal M_N(\psi^N)|\to0$ in $L^2(\bP_N)$.
\end{lem}

\begin{proof}
Recalling \eqref{eq:martingale} and \eqref{eq:martingale1}, it is not hard to get
\begin{equation}
  \begin{aligned}
    \big\langle\mathcal M_N(\psi^N)\big\rangle
    &= \int_0^T \sum_{j=1}^{N-1} \frac{p\eta_j+\sigma_N}N \left[ \sum_{i=K+1}^{N-K} \bar\psi_i^N\big(f(\hat\eta_i^{j,j+1}) - f(\hat\eta_i)\big) \right]^2dt\\
    &\le \frac{C\sigma_N}N\int_0^T \sum_{j=1}^{N-1} \left[ \sum_{i=K+1}^{N-K} \left| \bar\psi_i^N\big(\hat\eta_i^{j,j+1} - \hat\eta_i\big) \right| \right]^2dt.
  \end{aligned}
\end{equation}
Doob's inequality and \eqref{eq:exchange-smooth} then yield that
\begin{align}
  \bE_N \left[ \big|\mathcal M_N(\vf)\big|^2 \right] \le 4\bE_N \big[\big\langle\mathcal M_N(\vf)\big\rangle\big] \le C'\sigma_NK^{-2}.
\end{align}
This term vanishes since \eqref{eq:assp} and \eqref{eq:assp-k} contain $K^2 \gg N^{8/7} \gg \sigma_N$.
\end{proof}

\begin{lem}
\label{lem:esti-a-bd}
As $N\to\infty$, $|\cA_N(\psi^N)|\to0$ in $L^2(\bP_N)$.
\end{lem}

\begin{proof}
Recall that $g_i=\hat J_i-J(\hat\eta_i)$.
Applying \ref{eq:a1},
\begin{equation}
  \begin{aligned}
    \cA_N(\psi^N) =\;&\cA_{N,1}(\psi^N) + \cA_{N,2}(\psi^N)\\
    + \int_0^T &\big[\bar\psi_{K+1}^Nf'\big(\hat\eta_{K+1}\big)g_K - \bar\psi_{N-K+1}^Nf'\big(\hat\eta_{N-K+1}\big)g_{N-K}\big]\,dt,
  \end{aligned}
\end{equation}
where
\begin{equation}
  \begin{aligned}
    \cA_{N,1}(\psi^N) &:= \int_0^T \sum_{i=K+1}^{N-K} f'\big(\hat\eta_{i+1}\big)g_i\nabla\bar\psi_i^Ndt,\\
    \cA_{N,2}(\psi^N) &:= \int_0^T \sum_{i=K+1}^{N-K} \bar\psi_i^Ng_i\nabla f'\big(\hat\eta_i\big)dt.
  \end{aligned}
\end{equation}
The boundary integrals are negligible since $g_i$ is bounded and \eqref{eq:bd-estimate}:
\begin{align}
  \left| \int_0^T \big[\bar\psi_{K+1}^Nf'(\hat\eta_{K+1})g_K - \bar\psi_{N-K+1}^Nf'(\hat\eta_{N-K+1})g_{N-K}\big]\,dt \right| \le \frac{CK}{\sigma_N}.
\end{align}

The estimate for the remaining two terms is similar to Lemma \ref{lem:esti-a}.
Using \CS inequality,
\begin{align}
  \big|\cA_{N,1}(\psi^N)\big|^2 \le |f'|_\infty^2\int_0^T \sum_{i=K+1}^{N-K} g_i^2dt \int_0^T \sum_{i=K+1}^{N-K} \big|\nabla\bar\psi_i^N\big|^2dt.
\end{align}
Notice that $|\nabla\bar\psi_i^N| \le C\sigma_N^{-1}$.
By Proposition \ref{prop:one-block},
\begin{align}
  \bE_N \left[ \big|\cA_{N,1}(\psi^N)\big|^2 \right] \le C \left( \frac{K^2}{\sigma_N} + \frac NK \right) \frac N{\sigma_N^2} = C\left( 1 + \frac{N\sigma_N}{K^3} \right) \frac{NK^2}{\sigma_N^3}. \end{align}
Hence, we can conclude from \eqref{eq:assp} and \eqref{eq:assp-k}.
Similarly for $\cA_{N,2}$,
\begin{align}
  \big|\cA_{N,2}(\psi^N)\big|^2 \le |f''|_\infty^2\|\psi\|_{L^\infty}^2 \int_0^T \sum_{i=K+1}^{N-K} g_i^2dt \int_0^T \sum_{i=K+1}^{N-K} \big|\nabla\hat\eta_i\big|^2dt.
\end{align}
The estimate follows exactly the same as \eqref{eq:a2}.
\end{proof}

The limit of $\cS_N(\psi^N)$ differs essentially from that of $\cS_N(\vf)$ in Lemma \ref{lem:esti-s}.
We see below the divergent term mentioned at the beginning of this section.

\begin{lem}
\label{lem:esti-s-bd}
As $N\to\infty$, $\cS_N(\psi^N)$ satisfies that 
\begin{align}
  \lim_{N\to\infty} \bP_N \left\{ \cS_N(\psi^N) + \sigma_N\int_0^T \sum_{i=K+1}^{N-K} \nabla\big(\psi_if(\hat\eta_i)\big)\nabla\alpha_i^Ndt \le 0 \right\} = 1. 
\end{align}
\end{lem}

\begin{proof}
As in the previous lemma, we start from integration by parts.
From \eqref{eq:s1},
\begin{equation}
\label{eq:s-bd}
  \begin{aligned}
    \cS_N(\psi^N) = \cS_{N,1}(\psi^N) + \cS_{N,2}(\psi^N) - \sigma_N\int_0^T \bar\psi_{K+1}^Nf'\big(\hat\eta_{K+1}\big)\nabla\hat\eta_K\,dt\\
  +\;\sigma_N\int_0^T \bar\psi_{N-K+1}^Nf'\big(\hat\eta_{N-K+1}\big)\nabla\hat\eta_{N-K}\,dt,
  \end{aligned}
\end{equation}
where $\cS_{N,1}(\psi^N)$, $\cS_{N,2}(\psi^N)$ are defined by \eqref{eq:s2} with $\bar\vf_i$ replaced by $\bar\psi_i^N$.
To deal with the boundary terms, observe from \eqref{eq:bd-estimate} that
\begin{align}
  \bigg|\sigma_N\int_0^T \bar\psi_{K+1}^Nf'\big(\hat\eta_{K+1}\big)\nabla\hat\eta_K\,dt\bigg| \le K\|\psi\|_{L^\infty}|f'|_\infty \int_0^T \big|\nabla\hat\eta_K\big|dt.
\end{align}
Notice that the elementary estimate $|\nabla\hat\eta_K| \le K^{-1}$ is insufficient here since it leads to an upper bound of constant order.
Instead, we use Remark \ref{rem:bd-block} to obtain that
\begin{align}
\label{eq:s-bd1}
  \bE_N \left[ \bigg|\sigma_N\int_0^T \bar\psi_{K+1}^Nf'\big(\hat\eta_{K+1}\big)\nabla\hat\eta_K\,dt\bigg|^2 \right] \le C \left( \frac K{\sigma_N}+\frac1K \right),
\end{align}
which vanishes as $N\to\infty$.
The other boundary term is estimated similarly.

For the remaining two terms, we first look at $\mathcal S_{N,2}(\psi^N)$.
The convexity of $f$ assures that $\nabla\hat\eta_i\nabla f'(\hat\eta_i) \ge 0$ for all $i$, therefore,
\begin{align}
\label{eq:s-bd2}
  \mathcal S_{N,2}(\psi^N) = -\sigma_N\int_0^T \sum_{i=K+1}^{N-K} \bar\psi_i^N\nabla\hat\eta_i\nabla f'(\hat\eta_i)dt \le 0.
\end{align}
For $\cS_{N,1}(\psi^N)$, the following replacement holds as $|\nabla\bar\psi_i^N| \le C\sigma_N^{-1}$:
\begin{align}
\label{eq:s-bd3}
  \left| \cS_{N,1}(\psi^N) + \sigma_N\int_0^T \sum_{i=K+1}^{N-K} \nabla f(\hat\eta_i)\nabla\bar\psi_i^Ndt \right| \le \frac{CN}{K^2}. 
\end{align}
Plugging \eqref{eq:s-bd1}--\eqref{eq:s-bd3} into \eqref{eq:s-bd} and using \eqref{eq:assp}, \eqref{eq:assp-k}, 
\begin{align}
\label{eq:s-bd4}
  \lim_{N\to\infty} \bP_N \left\{ \cS_N(\psi^N) + \sigma_N\int_0^T \sum_{i=K+1}^{N-K} \nabla f(\hat\eta_i)\nabla\bar\psi_i^Ndt \le 0 \right\} = 1. 
\end{align}

Now we look closely at the integral appeared in \eqref{eq:s-bd4}.
Notice that
\begin{align}
  \nabla\bar\psi_i^N = N\int_{\frac iN-\frac1{2N}}^{\frac iN+\frac1{2N}} \left[ \psi \left( t,x+\frac1N \right) \alpha_N \left( x+\frac1N \right) - \psi(t,x)\alpha_N(x) \right] dx. 
\end{align}
With the notation $\bar\alpha_i^N := N\int \alpha_N(x)\chi_{i,N}(x)dx$, 
\begin{align}
\label{eq:psi1}
  \big|\nabla\bar\psi_i^N - \alpha_i^N\nabla\bar\psi_i - \psi_{i+1}\nabla\bar\alpha_i^N\big| \le \frac{2\|\partial_x\psi\|_{L^\infty}|\alpha'_N|_\infty}{N^2} \le \frac C{N\sigma_N}. 
\end{align}
Moreover, observe that for large $N$ and $i \le N/4$,
\begin{align}
\label{eq:psi2}
  \big|\nabla\bar\alpha_i^N-\nabla\alpha_i^N\big| \le CN^{-2}\max \left\{ |\alpha''_N(x)|; 0<x<\frac14 \right\} \le C\sigma_N^{-2}.
\end{align}
As $\delta_N=1/2+o_N(1)$, for $N/4 < i \le N\delta_N$,
\begin{align}
\label{eq:psi3}
  \big|\nabla\alpha_i^N-\nabla\bar\alpha_i^N\big| \le |\nabla\alpha_i^N| + |\nabla\bar\alpha_i^N| \le 2e^{-N(4\sigma_N)^{-1}} \le C\sigma_N^{-2}. 
\end{align}
By \eqref{eq:psi1}--\eqref{eq:psi3} and similar bound for $i>N\delta_N$, 
\begin{equation}
  \begin{aligned}
    &\sigma_N \left| \int_0^T \sum_{i=K+1}^{N-K} \nabla f(\hat\eta_i)\big(\nabla\bar\psi_i^N - \alpha_i^N\nabla\bar\psi_i - \psi_{i+1}\nabla\alpha_i^N\big)dt \right|\\
    \le\;&CN\sigma_N \cdot |f'|_\infty\sup_i \big|\nabla\hat\eta_i\big| \cdot \left( \frac1{N\sigma_N}+\frac1{\sigma_N^2} \right) \le \frac{CN}{\sigma_NK}.
  \end{aligned}
\end{equation}
Therefore, we are allowed to replace $\nabla\bar\psi_i^N$ in \eqref{eq:s-bd4} by $\alpha_i^N\nabla\bar\psi_i+\psi_{i+1}\nabla\alpha_i^N$.

To complete the proof, we need to compute that
\begin{equation}
  \begin{aligned}
    \sum_{i=K+1}^{N-K} \big(\alpha_i^N\nabla\bar\psi_i + \psi_{i+1}\nabla\alpha_i^N\big)\nabla f(\hat\eta_i) - \sum_{i=K+1}^{N-K} \nabla\big(\psi_if(\hat\eta_i)\big)\nabla\alpha_i^N \\
    = \alpha_i^N\nabla\bar\psi_if(\hat\eta_{i+1})\Big|_{i=K}^{N-K} + \sum_{i=K+1}^{N-K} f(\hat\eta_i) \big[\nabla^*(\alpha_i^N\nabla\bar\psi_i) - \nabla\psi_i\nabla\alpha_i^N\big].
  \end{aligned}
\end{equation}
By \eqref{eq:bd-estimate}, the contribution of the boundary terms is negligible:
\begin{align}
\label{eq:s-bd5}
  \sigma_N \left| \int_0^T \alpha_i^N\nabla\bar\psi_if(\hat\eta_{i+1})\Big|_{i=K}^{N-K}dt \right| \le C\sigma_N\frac K{\sigma_N}\frac1N \le \frac{CK}N.
\end{align}
Since $|\nabla^*(\alpha_i^N\nabla\bar\psi_i)| \le |\alpha_{i-1}^N\Delta\bar\psi_i| + |\nabla\alpha_{i-1}^N\nabla\bar\psi_i| \le CN^{-2}+|\nabla\alpha_{i-1}^N\nabla\bar\psi_i|$,
\begin{equation}
\label{eq:s-bd6}
  \begin{aligned}
    &\sigma_N \left| \int_0^T \sum_{i=K+1}^{N-K} f(\hat\eta_i) \big[\nabla^*(\alpha_i^N\nabla\bar\psi_i) - \nabla\psi_i\nabla\alpha_i^N\big]\,dt \right| \\
    \le\;&\frac{C\sigma_N}N + \sigma_N\int_0^T \sum_{i=K+1}^{N-K} f(\hat\eta_i)\big(|\nabla\alpha_{i-1}^N\nabla\bar\psi_i| + |\nabla\psi_i\nabla\alpha_i^N|\big)dt.
  \end{aligned}
\end{equation}
Finally, the definition of $\alpha_N$ yields that
\begin{align}
  \sum_{i=K+1}^{N-K} f(\hat\eta_i)\big|\nabla\alpha_{i-1}^N\nabla\bar\psi_i\big| \le \frac{|f|_\infty\|\partial_x\psi\|_{L^\infty}}N \sum_{i=K+1}^{N-K} \big|\nabla\alpha_{i-1}^N\big| \le \frac CN.
\end{align}
As the same bound holds for $\nabla\psi_i\nabla\alpha_i^N$,
\begin{align}
\label{eq:s-bd7}
  \sigma_N\int_0^T \sum_{i=K+1}^{N-K} f(\hat\eta_i)\big(|\nabla\alpha_{i-1}^N\nabla\bar\psi_i| + |\nabla\psi_i\nabla\alpha_i^N|\big)dt \le \frac{C\sigma_N}N.
\end{align}
The desired conclusion then follows by putting \eqref{eq:s-bd5}--\eqref{eq:s-bd7} together.
\end{proof}

We are left with the last term $\mathcal B_N(\psi,\alpha_N)$ in \eqref{eq:bd-ent-prod-decom}.
The next lemma shows that $\mathcal B_N(\psi,\alpha_N)$ cancels exactly the divergent term in Lemma \ref{lem:esti-s-bd}.

\begin{lem}
\label{lem:esti-b}
As $N\to\infty$, $\mathcal B_N(\psi,\alpha_N)$ satisfies that
\begin{equation}
  \begin{aligned}
    \lim_{N\to\infty} \bP_N \bigg\{&\mathcal B_N(\psi,\alpha_N) - \sigma_N\int_0^T \sum_{i=K+1}^{N-K} \nabla \left( \psi_if(\hat\eta_i) \right) \nabla\alpha_i^Ndt\\
    &\le \int_0^T f(\rho_-)\psi(\cdot,0)dt + \int_0^T f(\rho_+)\psi(\cdot,1)dt\bigg\} = 1.
  \end{aligned}
\end{equation}
\end{lem}

The following boundary block estimate is necessary for Lemma \ref{lem:esti-b}.
It is proved similarly to Proposition \ref{prop:one-block} and \ref{prop:h1}.
We postpone it to Section \ref{subsec:block}.

\begin{prop}[Boundary one-block estimates]
\label{prop:bd-one-block}
With some constant $C$,
\begin{align}
  \bE_N \left[ \int_0^T \left( \big|\hat\eta_K-\rho_-(t)\big|^2 + \big|\hat\eta_{N-K}-\rho_+(t)\big|^2 \right) dt \right] \le C \left( \frac K{\sigma_N} + \frac1{\wts_N} + \frac1K \right),
\end{align}
where $\rho_\pm=\rho_\pm(t)$ are defined in \eqref{eq:cl2}.
\end{prop}

\begin{proof}[Proof of Lemma \ref{lem:esti-b}]
First notice from \eqref{eq:esti-ent} that
\begin{align}
\label{eq:b1}
  \mathcal B_N(\psi,\alpha_N) = \int_0^T \sum_{i=K+1}^{N-K} q(\hat\eta_i)\psi_i\nabla\alpha_i^Ndt \le \int_0^T \sum_{i=K+1}^{N-K} f(\hat\eta_i)\psi_i\big|\nabla\alpha_i^N\big|dt.
\end{align}
Recall that $\delta_N=1/2+O(\sigma_N^{-1})$ in \eqref{eq:auxiliary}.
For large $N$, choose the integer $j_N:=[N\delta_N+1/2]$, then $j_N$ satisfies that
\begin{align}
  \delta_N \in \left[ \frac{j_N}N-\frac1{2N}, \frac{j_N}N+\frac1{2N} \right).
\end{align}
By the definition of $\alpha_N$ we can rewrite \eqref{eq:b1} as\footnote{The term with $i=j_N$ is omitted since it is vanishing in the limit $N\to\infty$.},
\begin{align}
\label{eq:b2}
  \mathcal B_N(\psi,\alpha_N) \le \int_0^T \sum_{i=K+1}^{j_N-1} f(\hat\eta_i)\psi_i\nabla\alpha_i^Ndt - \int_0^T \sum_{i=j_N+1}^{N-K} f(\hat\eta_i)\psi_i\nabla\alpha_i^Ndt.
\end{align}
We calculate the two terms in \eqref{eq:b2} respectively.
For the first one,
\begin{align*}
  &\sum_{i=K+1}^{j_N-1} f(\hat\eta_i)\psi_i\nabla\alpha_i^N = \sum_{i=K+1}^{j_N-1} f(\hat\eta_i)\psi_i \nabla^*\big(1-\alpha_{i+1}^N\big) \\
  =\;&f(\hat\eta_{K+1})\psi_{K+1}\big(1-\alpha_{K+1}^N\big) - f(\hat\eta_{j_N})\psi_{j_N}\big(1-\alpha_{j_N}^N\big) + \sum_{i=K+1}^{j_N-1} \nabla\big(f(\hat\eta_i)\psi_i\big)(1-\alpha_{i+1}^N) \\
  \le\;&f(\hat\eta_{K+1})\psi_{K+1} + \sum_{i=K+1}^{j_N-1} \nabla\big(f(\hat\eta_i)\psi_i\big)\big(1-\alpha_{i+1}^N\big). 
\end{align*}
The second one follows similarly and we obtain that 
\begin{align*}
  \sum_{i=j_N+1}^{N-K} f(\hat\eta_i)\psi_i\nabla\alpha_i^N \ge -f(\hat\eta_{N-K+1})\psi_{N-K+1} + \sum_{i=j_N+1}^{N-K} \nabla\big(f(\hat\eta_i)\psi_i\big)\big(1-\alpha_{i+1}^N\big). 
\end{align*}
From these estimates and \eqref{eq:b2}, we obtain that
\begin{equation}
\label{eq:b3}
  \begin{aligned}
    \mathcal B_N(\psi,\alpha_N) \le \int_0^T \big[f(\hat\eta_{K+1})\psi_{K+1} + f(\hat\eta_{N-K+1})\psi_{N-K+1}\big]dt \\
    + \int_0^T \sum_{i=K+1}^{N-K} \sgn(j_N-i)\nabla\big(f(\hat\eta_i)\psi_i\big)\big(1-\alpha_{i+1}^N\big)dt. 
  \end{aligned}
\end{equation}
The definition of $\alpha_N$ assures that 
\begin{align}
  \nabla\alpha_i^N = 
  \begin{cases}
    \sigma_N^{-1}(1-\alpha_{i+1}^N), &1 \le i \le j_N-1, \\
    -\sigma_N^{-1}(1-\alpha_{i+1}^N), &j_N+1 \le i \le N-1. 
  \end{cases}
\end{align}
Therefore, \eqref{eq:b3} can be rewritten as
\begin{equation}
  \begin{aligned}
    \mathcal B_N(\psi,\alpha_N) - \sigma_N\int_0^T \sum_{i=K+1}^{N-K} \nabla\big(f(\hat\eta_{i-1})\psi_{i-1}\big)\nabla\alpha_i^Ndt\\
    \le \int_0^T \big[f(\hat\eta_{K+1})\psi_{K+1} + f(\hat\eta_{N-K+1})\psi_{N-K+1}\big]dt.
  \end{aligned}
\end{equation}
Sending $N\to\infty$ and applying Proposition \ref{prop:bd-one-block}, we conclude the result. 
\end{proof}

\section{Estimates}
In this section we establish and prove the technical results used in this article.

\subsection{A priori estimate on boundary currents}

Recall that $\eta(t)$ is the process generated by $NL_{N,t}$ for the generator $L_{N,t}$ in \eqref{eq:gen}.
Define the microscopic currents $j_{i,i+1}$ associated to $L_{N,t}$ through the conservation law $L_{N,t} [\eta_i] = j_{i-1,i}-j_{i,i+1}$.
They are equal to
\begin{align}
\label{eq:current0}
  j_{i,i+1} =
  \begin{cases}
    \wts_N[\alpha(t)-(\alpha(t)+\gamma(t))\eta_1], &i=0,\\
    p\eta_i(1-\eta_{i+1})+\sigma_N(\eta_i-\eta_{i+1}), &1 \le i \le N-1,\\
    \wts_N[(\beta(t)+\delta(t))\eta_N-\delta(t)], &i=N.
  \end{cases}
\end{align}
Following the argument in \cite[Section 2]{DeMasiO20}, for $i=1$, ..., $N-1$ define the counting processes associated to the process $\{\eta(t)\}_{t\ge0}$ by
\begin{equation}
  \begin{aligned}
    h_+(i,t) &:= \text{number of jumps}\ i \to i+1\ \text{in}\ [0,t],\\
    h_-(i,t) &:= \text{number of jumps}\ i+1 \to i\ \text{in}\ [0,t].
  \end{aligned}
\end{equation}
These definitions extend to the boundaries $i=0$ and $i=N$ as
\begin{equation}
  \begin{aligned}
    h_+(0,t) &:= \text{number of particles created at}\ 1\ \text{in}\ [0,t],\\
    h_-(0,t) &:= \text{number of particles annihilated at}\ 1\ \text{in}\ [0,t],\\
    h_+(N,t) &:= \text{number of particles annihilated at}\ N\ \text{in}\ [0,t],\\
    h_-(N,t) &:= \text{number of particles created at}\ N\ \text{in}\ [0,t].
  \end{aligned}
\end{equation}
With $h(i,t):=h_+(i,t)-h_-(i,t)$, the conservation law is microscopically given by
\begin{align}
\label{eq:cl0}
  \eta_i(t)-\eta_i(0)=h(i-1,t)-h(i,t), \quad \forall\,i=1,\ldots,N.
\end{align}
Furthermore, for $i=0$, ..., $N$, there is a martingale $M_i(t)$ such that
\begin{align}
  h(i,t) = N\int_0^t j_{i,i+1}(s)ds + M_i(t).
\end{align}
As $\eta_i(t)\in\{0,1\}$, \eqref{eq:cl0} yields that $|h(i,t)-h(i',t)|\le|i-i'|$.
Therefore,
\begin{align}
  \left| \bE_N \left[ \int_0^t j_{i,i+1}(s)ds \right] - \bE_N \left[ \int_0^t j_{i',i'+1}(s)ds \right] \right| \le \frac{|i-i'|}N \le 1
\end{align}
for all $i$, $i'=0$, ..., $N$.
In particular, for the fixed time $T>0$,
\begin{align}
\label{eq:current1}
  \left| \bE_N \left[ \int_0^T j_{N,N+1}(t)dt \right] - \bE_N \left[ \int_0^T \frac1{N-1}\sum_{i=1}^{N-1} j_{i,i+1}(t)dt \right] \right| \le 1.
\end{align}
Since $j_{i,i+1}=p\eta_i(1-\eta_{i+1})+\sigma_N(\eta_i-\eta_{i+1})$,
\begin{equation}
  \begin{aligned}
    \left| \frac1{N-1}\sum_{i=1}^{N-1} j_{i,i+1}(t) \right|
    &= \left| \frac{\sigma_N}{N-1}(\eta_1-\eta_N) + \frac1{N-1}\sum_{i=1}^{N-1} p\eta_i(1-\eta_{i+1})\right|\\
    &\le \sigma_N(N-1)^{-1}+p = o(1)+p, \quad N\to\infty.
  \end{aligned}
\end{equation}
Hence, from \eqref{eq:current1} we see that
\begin{align}
  \limsup_{N\to\infty} \left| \bE_N \left[ \int_0^T j_{N,N+1}(t)dt \right] \right| \le 1+p.
\end{align}
Recalling the definitions of $j_{N,N+1}$ in \eqref{eq:current0}, we obtain a priori bound:
\begin{align}
\label{eq:priori-bd}
  \limsup_{N\to\infty} \left| \bE_N \left[ \int_0^T \big(\beta(t)+\delta(t)\big)\big(\eta_N-\rho_+(t)\big)dt \right] \right| \le \frac C{\wts_N},
\end{align}
with $\rho_+$ given by \eqref{eq:cl2}.
It is easy to obtain the same bound for $\eta_1-\rho_-(t)$.

\subsection{Dirichlet forms}
\label{subsec:dir}

Recall the process $\eta(\cdot)\in\Omega_N=\{0,1\}^N$ generated by $NL_{N,t}$.
Denote by $\mu_{N,t}$ the distribution of $\eta(t)$.
Define the Dirichlet form
\begin{align}
\label{eq:dir-exc}
  \fD_{\exc,N}(t) := \frac12\sum_{\eta\in\Omega_N} \sum_{i=1}^{N-1} 
    \left(\sqrt{\mu_{N,t}(\eta^{i,i+1})} - \sqrt{\mu_{N,t}(\eta)}\right)^2. 
\end{align}
Denote by $\nu_{\pm,t}$ the Bernoulli measure with densities $\rho_\pm(t)$ given in \eqref{eq:cl2}:
\begin{equation}
\label{eq:bbern}
  \nu_{\pm,t}(\eta) = \prod_{i=1}^N\,\rho_\pm(t)^{\eta_i}\big[1-\rho_\pm(t)\big]^{1-\eta_i}, \quad \forall\,\eta \in \Omega_N. 
\end{equation}
The boundary Dirichlet forms are defined as
\begin{equation}
\label{eq:dir-bd}
  \begin{aligned}
    &\fD_{-,N}(t) := \frac12\sum_{\eta\in\Omega_N} \rho_-^{1-\eta_1}(1-\rho_-)^{\eta_1}
    \left(\sqrt{f_{N,t}^-(\eta^1)} - \sqrt{f_{N,t}^-(\eta)}\right)^2\nu_{-,t}(\eta), \\
    &\fD_{+,N}(t) := \frac12\sum_{\eta\in\Omega_N} \rho_+^{1-\eta_N}(1-\rho_+)^{\eta_N}
    \left(\sqrt{f_{N,t}^+(\eta^N)} - \sqrt{f_{N,t}^+(\eta)}\right)^2\nu_{+,t}(\eta), 
  \end{aligned}
\end{equation}
where $f^\pm_{N,t}:=\mu_{N,t}/\nu_{\pm,t}$ is the probability density function.

\begin{prop}
\label{prop:dir}
There exists $C$ independent of $N$ such that
\begin{align}
\label{eq:dir}
  \int_0^T \fD_{\exc,N}(t)dt \le \frac C{\sigma_N}, \quad \int_0^T \big[\fD_{-,N}(t)+\fD_{+,N}(t)\big]dt \le \frac C{\wts_N}.
\end{align}
\end{prop}

\begin{proof}
Define the relative entropy $H_{\pm,N}(t)$ by
\begin{align}
  H_{\pm,N}(t) := \sum_{\eta\in\Omega_N} f_{N,t}^\pm(\eta)\log f_{N,t}^\pm(\eta)\nu_{\pm,t}(\eta).
\end{align}
Standard manipulation with Kolmogorov equation gives that
\begin{equation}
\label{eq:h1}
  \begin{aligned}
    \frac1N\frac d{dt}H_{-,N}(t)
    &= \sum_{\eta\in\Omega_N} f_{N,t}^-L_{N,t} \big[\log f_{N,t}^-\big]\nu_{-,t} - \frac1N\sum_{\eta\in\Omega_N} f_{N,t}^-\frac d{dt}\nu_{-,t}\\
    &\le \sum_{\eta\in\Omega_N} f_{N,t}^-L_{N,t} \big[\log f_{N,t}^-\big]\nu_{-,t} + C.
  \end{aligned}
\end{equation}

We divide the summation in the right-hand side of \eqref{eq:h1} into two parts:
\begin{align}
  \Gamma_{N,t}^{-,(1)} &:= \sum_{\eta\in\Omega_N} f_{N,t}^- \big(L_{N,t} - \wts_NL_{+,t}\big)\big[\log f_{N,t}^-\big]\nu_{-,t},\\
  \Gamma_{N,t}^{-,(2)} &:= \wts_N\sum_{\eta\in\Omega_N} f_{N,t}^-L_{+,t} \big[\log f_{N,t}^-\big]\nu_{-,t}.
\end{align}
Exploiting the inequality $x(\log y-\log x)\le2\sqrt x(\sqrt y-\sqrt x)$ for $x$, $y>0$,
\begin{equation}
  \begin{aligned}
    \Gamma_{N,t}^{-,(1)}
    &= \sum_{\eta\in\Omega_N} f_{N,t}^- \big(p\Lta+\sigma_N\Ls+\wts_NL_{-,t}\big)\big[\log f_{N,t}^-\big]\nu_{-,t}\\
    &\le 2\sum_{\eta\in\Omega_N} g_{N,t}^- \big(p\Lta+\sigma_N\Ls+\wts_NL_{-,t}\big)\big[g_{N,t}^-\big]\nu_{-,t},
  \end{aligned}
\end{equation}
where we denote $g_{N,t}^-:=(f_{N,t}^-)^{1/2}$.
Since the measure $\nu_{-,t}$ is invariant under $\Lta$, $\Ls$ as well as $L_{-,t}$, direct computation shows that
\begin{align}
  \Gamma_{N,t}^{-,(1)} \le -\wts_N\big(\alpha(t)+\gamma(t)\big)\fD_{-,N}(t) - \sigma_N\fD_{\exc,N}(t) + C.
\end{align}
Meanwhile, since $f_{N,t}^-\nu_{-,t}=f_{N,t}^+\nu_{+,t}$,
\begin{equation}
  \begin{aligned}
    \Gamma_{N,t}^{-,(2)} = \wts_N\sum_{\eta\in\Omega_N} f_{N,t}^-L_{+,t} \left[ \log f_{N,t}^++\log \left( \frac{\nu_{+,t}}{\nu_{-,t}} \right) \right] \nu_{-,t}\\
    \le -\wts_N\big(\beta(t)+\delta(t)\big)\fD_{+,N}(t) + \wts_N\sum_{\eta\in\Omega_N} f_{N,t}^-L_{+,t} \left[ \log \left( \frac{\nu_{+,t}}{\nu_{-,t}} \right) \right] \nu_{-,t}.
  \end{aligned}
\end{equation}
By the definition of $\nu_{\pm,t}$ and $L_{-,t}$,
\begin{align}
  L_{-,t} \left[ \log \left( \frac{\nu_{+,t}}{\nu_{-,t}} \right) \right] = \mathfrak R(t) \big(\beta(t)+\delta(t)\big) \big(\eta_N-\rho_+(t)\big),
\end{align}
where
\begin{align}
  \mathfrak R(t) := \log \left[ \frac{\rho_-(t)}{1-\rho_-(t)} \right] - \log \left[ \frac{\rho_+(t)}{1-\rho_+(t)} \right].
\end{align}
Applying the priori estimate \eqref{eq:priori-bd},
\begin{align}
  \int_0^T \Gamma_{N,t}^{-,(2)}dt \le -\wts_N\int_0^T \big(\beta(t)+\delta(t)\big)\fD_{+,N}(t)dt + C.
\end{align}
Integrating \eqref{eq:h1} over $[0,T]$ and using the results we proved above,
\begin{equation}
  \begin{aligned}
    \frac{H_{-,N}(T)-H_{-,N}(0)}N \le C - \sigma_N\int_0^T \fD_{\exc,N}(t)dt\\
    -\,\wts_N\int_0^T \Big[\big(\alpha(t)+\gamma(t)\big)\fD_{-,N}(t) + \big(\beta(t)+\delta(t)\big)\fD_{+,N}(t)\Big] dt.
  \end{aligned}
\end{equation}
We can conclude \eqref{eq:dir} since $\alpha$, $\beta$, $\gamma$, $\delta$ are uniformly bounded from $0$.
\end{proof}

\subsection{Logarithmic Sobolev inequalities}
\label{subsec:log-sobol-ineq}

For $\rho \in (0,1)$, let $\nu_\rho$ be the product Bernoulli measure on $\Omega_k=\{0,1\}^k$ with homogeneous density $\rho$.
The micro canonical configuration space $\Omega_{k,\rho_*}$ is defined for $\rho_*=0$, $1/k$, ..., $(k-1)/k$, $1$ as
\begin{align}
\label{eq:mic-suf}
  \Omega_{k,\rho_*} := \left\{\eta=(\eta_1,\dots,\eta_k) \in \Omega_k~\bigg|~\frac1k\sum_{i=1}^k \eta_i = \rho_*\right\}.
\end{align}
Let $\tilde\nu^k(\eta|\rho_*)$ be the uniform measure on $\Omega_{k,\rho_*}$.
The log-Sobolev inequality for the simple exclusion (\cite{Yau97}) yields that 
there exists a universal constant $C_\LS$ such that
\begin{equation}
\label{eq:logsob}
  \begin{aligned}
    &\sum_{\eta\in\Omega_{k,\rho_*}} f(\eta)\log f(\eta)\nu^k(\eta|\rho_*)\\
    \le\;&\frac{C_\LS k^2}2 \sum_{\eta\in\Omega_{k,\rho_*}} \sum_{i=1}^{k-1} \left( \sqrt f (\eta^{i,i+1}) - \sqrt f (\eta) \right)^2 \nu^k(\eta|\rho_*).
  \end{aligned}
\end{equation}
for any $f\ge0$ such that $\sum_{\eta\in\Omega_{k,\rho_*}} f(\eta)\nu(\eta|\rho_*) = 1$.
In \cite[Proposition A.1]{DMOX22}, \eqref{eq:logsob} is extended to a log-Sobolev inequality associated to $\nu_\rho$.

\begin{prop}\label{prop:logsob-bound}
There exists a constant $C_\rho$, such that
\begin{equation}
\label{eq:logsob-bound}
  \begin{aligned}
    \sum_{\eta\in\Omega_k} f(\eta)\log f(\eta) \nu_\rho(\eta)
    \le \frac{C_\rho k^2}2 \sum_{\eta\in\Omega_k}
    \sum_{i=1}^{k-1} \left( \sqrt f (\eta^{i,i+1}) - \sqrt f (\eta) \right)^2 \nu_\rho(\eta)\\
    + \frac{C_\rho k}2 \sum_{\eta\in\Omega_k}
    \rho^{1-\eta_1}(1-\rho)^{\eta_1} \left( \sqrt f (\eta^1) - \sqrt f (\eta) \right)^2 \nu_\rho(\eta),
  \end{aligned}
\end{equation}
for any $f\ge0$ such that $\sum_{\eta\in\Omega_k} f(\eta)\nu_\rho(\eta) = 1$.
\end{prop}

We refer to \cite[Appendix A]{DMOX22} for the details of the proof.

\subsection{Block estimates}
\label{subsec:block}

Now we state the proofs of Proposition \ref{prop:one-block}, \ref{prop:h1} and \ref{prop:bd-one-block}.
They are based on Proposition \ref{prop:dir}, \eqref{eq:logsob} and \eqref{eq:logsob-bound}.
Since the size of the block varies in the proofs, we keep the subscription $K$ in $\bar\eta_{i,K}$, $\hat\eta_{i,K}$ and $\hat J_{i,K}$.

\begin{proof}[Proof of Proposition \ref{prop:one-block}]
Recall the space $\Omega_{k,\rho^*}$ defined in \eqref{eq:mic-suf} and the uniform measure $\nu^k(\cdot|\rho_*)$ on it.
Also define 
\begin{equation}
  \begin{aligned}
    &\bar\mu_{N,t}^{i,k}(\rho_*) := \mu_{N,t}\left\{(\eta_{i-k+1},\ldots,\eta_i)\in\Omega_{k,\rho_*}\right\}, \\
    &\mu_{N,t}^{i,k}(\eta_{i-k+1},\ldots,\eta_i|\rho_*) := \frac{\mu_{N,t}(\eta_{i-k+1},\ldots,\eta_i)}{\bar\mu_{N,t}^{i,k}(\rho_*)}. 
  \end{aligned}
\end{equation}

Recall the block averages $\bar\eta_{i,K}$, $\hat\eta_{i,K}$ and $\hat J_{i,K}$ defined in \eqref{eq:block}, \eqref{eq:smooth-block} and \eqref{eq:current}.
Note that
\begin{align}
  \left[ \hat J_{i,K}-J(\hat\eta_{i,K}) \right]^2 \le 2 \left[ \hat J_{i,K}-J(\bar\eta_{i+K,2K}) \right]^2 + 2 \big[J(\hat\eta_{i,K})-J(\bar\eta_{i+K,2K})\big]^2.
\end{align}
We prove the estimate for the two terms respectively.
Denote $h_i=\hat J_{i,K}-J(\bar\eta_{i+K,2K})$. 
The relative entropy inequality yields that 
\begin{equation}
\label{eq:rel-ent-ineq}
  \begin{aligned}
    \int_{\Omega_{2K,\rho_*}} h_i^2d\mu_{N,t}^{i+K,2K}(\,\cdot\,|\rho_*) \le \frac1{aK} \bigg[ H\left(\mu_{N,t}^{i+K,2K}(\,\cdot\,|\rho_*); \nu^{2K}(\,\cdot\,|\rho_*)\right)\\
    +\,\log\int e^{aKh_i^2}d\nu^{2K}(\,\cdot\,|\rho_*) \bigg], \quad \forall a>0,
  \end{aligned}
\end{equation}
where $H$ is the relative entropy for two probability measures $\mu$ and $\nu$:
\begin{align}
  H(\mu;\nu) := \int (\log\mu - \log\nu)d\mu.
\end{align}
The log-Sobolev inequality \eqref{eq:logsob} yields that with some constant $C$,
\begin{align}\label{eq:logsob1}
    H\left(\mu_{N,t}^{i+K,2K}(\,\cdot\,|\rho_*); \nu^{2K}(\,\cdot\,|\rho_*)\right) \le CK^2\fD_{N,\rho_*}^{i+K,2K}(t),
\end{align}
where the Dirichlet form in the right-hand side is defined as
\begin{align}
  \fD_{N,\rho_*}^{i,k}(t) := \frac12\sum_{\eta\in\Omega_{k,\rho_*}}
  \sum_{i'=1}^{k-1} \left(\sqrt{\mu_{N,t}^{i,k}(\eta^{i',i'+1}|\rho_*)}
  - \sqrt{\mu_{N,t}^{i,k}(\eta|\rho_*)}\right)^2. 
\end{align}
For the uniform integral, with standard argument we can show that
\begin{align}
\label{eq:exp}
  \log\int e^{aKh_i^2}d\nu^{2K}(\,\cdot\,|\rho_*) \le C,
\end{align}
for some sufficiently small but fixed $a$.
From \eqref{eq:rel-ent-ineq}, \eqref{eq:logsob1} and \eqref{eq:exp},
\begin{align}
  \int_{\Omega_{2K,\rho_*}} h_i^2d\mu_{N,t}^{i+K,2K}(\,\cdot\,|\rho_*) \le CK\fD_{N,\rho_*}^{i+K,2K}(t)+\frac CK,
\end{align}
for all $i=K$, ..., $N-K$ and $\rho_*$.
Integrating the estimate above with respect to $\bar\mu_{N,t}^{i+K,2K}(\rho_*)$ and summing up in $i$, we obtain that
\begin{equation}
  \begin{aligned}
    \sum_{i=K}^{N-K} \int_{\Omega_N} h_i^2d\mu_{N,t}
    &\le CK\sum_{i=K}^{N-K} \sum_{\rho_*} \bar\mu_{N,t}^{i,k}(\rho_*)\fD_{N,\rho_*}^{i+K,2K}(t) + \frac{CN}K\\
    &\le C'K^2\fD_{\exc,N}(t) + \frac{CN}K.
  \end{aligned}
\end{equation}
The desired estimate for $h_i$ then follows from Proposition \ref{prop:dir}.
The upper bound for $J(\hat\eta_{i,K})-J(\bar\eta_{i+K,2K})$ can be obtained by exactly the same argument.
\end{proof}

\begin{proof}[Proof of Proposition \ref{prop:h1}]
Observe that for $i=K$, $K+1$, ..., $N-K$, 
\begin{align}
  \nabla\hat\eta_{i,K} = \hat\eta_{i+1,K}-\hat\eta_{i,K} = \frac{\bar\eta_{i+K,K}-\bar\eta_{i,K}}{K}. 
\end{align}
Applying the same argument used in Proposition \ref{prop:one-block}, 
\begin{align}
    \bE_N \left[ \int_0^T \sum_{i=K}^{N-K} (\bar\eta_{i+K,K}-\bar\eta_{i,K})^2dt \right] \le C \left( \frac{K^2}{\sigma_N} + \frac NK \right). 
\end{align}
The conclusion then follows directly. 
\end{proof}

\begin{rem}
\label{rem:bd-block}
With the arguments in the previous proofs, we also get that
\begin{align}
  \bE_N \left[ \int_0^T \big|\nabla\hat\eta_{i,K}\big|^2dt \right] \le C \left( \frac1{\sigma_NK} + \frac1{K^3} \right),
\end{align}
for each $i=K$, $K+1$, ..., $N-K$.
\end{rem}

\begin{proof}[Proof of Proposition \ref{prop:bd-one-block}]
Recall that $\nu_\rho$ is the product Bernoulli measure with rate $\rho$.
To shorten the notation, let $\alpha_{N,t}$ be the marginal distribution of $\{\eta_1, \eta_2, \ldots, \eta_{2K-1}\}$ under $\mu_{N,t}$.
By the relative entropy inequality, for any $a>0$,
\begin{equation}
  \begin{aligned}
    \int_{\Omega_N} \big|\hat\eta_{K,K}-\rho_-(t)\big|^2d\mu_{N,t} \le\;&\frac1{aK} \bigg[ H\big(\alpha_{N,t};\nu_{\rho_-(t)}\big)\\
    &+ \log\int e^{aK|\hat\eta_{K,K}-\rho_-(t)|^2}d\nu_{\rho_-(t)} \bigg].
  \end{aligned}
\end{equation}
Applying Proposition \ref{prop:logsob-bound} on the box $\{\eta_1,\ldots,\eta_{2K-1}\}$,
\begin{align}
  H\big(\alpha_{N,t};\nu_{\rho_-(t)}\big) 
  \le CK^2\fD_{\exc,N}(t) + CK\fD_{-,N}(t),
\end{align}
where 
$\fD_{\exc,N}(t)$, $\fD_{-,N}(t)$ are the Dirichlet forms defined in \eqref{eq:dir-exc} and \eqref{eq:dir-bd}.
On the other hand, the canonical integral as  can be bound similarly to \eqref{eq:exp} as
\begin{align}
  \log\int e^{aK|\hat\eta_{K,K}-\rho_-(t)|^2}d\nu_{\rho_-(t)} \le C, 
\end{align}
for sufficiently small $a$.
Therefore,
\begin{align}
  \int_{\Omega_N} \big|\hat\eta_{K,K}-\rho_-(t)\big|^2d\mu_{N,t} \le C \left( K\fD_{\exc,N}(t)+\fD_{-,N}(t)+\frac1K \right).
\end{align}
The estimate for $\hat\eta_{N-K+1,K}$ is proved in the same way.
Finally, in order to close the proof we only need to apply Proposition \ref{prop:dir}.
\end{proof}


\titleformat{\section}[hang]	
{\bfseries\large}{}{0em}{}[]
\titlespacing*{\section}{0em}{2em}{1.5em}

\section{References}	
\renewcommand{\section}[2]{}



\vspace{2em}
\noindent{\large Lu \textsc{Xu}}

\vspace{0.5em}
\noindent Gran Sasso Science Institute\\
Viale Francesco Crispi 7, L'Aquila, Italy\\
{\tt lu.xu@gssi.it}

\end{document}